\documentclass[10pt]{article}

\usepackage{amsmath,amssymb,amscd,amsthm,makeidx,txfonts,graphicx,url,bm}
\usepackage{a4wide,pstricks,xy}

\usepackage{youngtab}
\numberwithin{equation}{subsection}

\xyoption{all}
\input{xypic}

\newtheorem{theorem}{Theorem}[subsection]
\newtheorem{proposition}[theorem]{Proposition}
\newtheorem{corollary}[theorem]{Corollary}
\newtheorem{conjecture}[theorem]{Conjecture}
\newtheorem{lemma}[theorem]{Lemma}

\theoremstyle{remark}
\newtheorem{remark}[theorem]{Remark}

\theoremstyle{definition}

\def\beq{\begin{eqnarray}}
\def\eeq{\end{eqnarray}}
\def\bes{\begin{eqnarray*}}
\def\ees{\end{eqnarray*}}

\def\omhat{{\bm\omega}}

\def\muhat{{\bm \mu}}

\def\betahat{{\bm \beta}}

\def\lambdahat{{\bm \lambda}}
\def\alphahat{{\bm \alpha}}

\def\oP{\overline{\calP}}
\def\oT{\overline{\bT}}
\def\otT{\overline{\bT}{^o}}

\def\bT{{\mathbf{T}}}
\def\tT{{\bT}^o}
\def\otTp{\overline{\bf T}{^o_+}}

\def\bbV{\mathbb{V}}
\def\bbU{\mathbb{U}}
\def\C{\mathbb{C}}
\def\M{{\mathcal{M}}}

\def\calR{{\mathcal{R}}}
\def\calQ{{\mathcal{Q}}}

\def\calX{{\mathcal{X}}}
\def\calU{{\mathcal{U}}}
\def\calE{{\mathcal{E}}}
\def\calV{{\mathcal{V}}}

\def\calS{{\mathfrak{S}}}

\def\calZ{{\mathcal{Z}}}
\def\calP{\mathcal{P}}
\def\calH{\mathcal{H}}

\def\x{\mathbf{x}}
\def\y{\mathbf{y}}
\def\c{\mathbf{c}}

\def\v{\mathbf{v}}

\def\e{\mathbf{e}}
\def\w{\mathbf{w}}

\def\P{\mathcal{P}}

\def\tomega{{{\omega}^o}}
\def\tomhat{{{\omhat}^o}}

\def\N{\mathbb{Z}_{\geq 0}}

\def\F{\mathbb{F}}
\def\Q{\mathbb{Q}}

\def\t{\mathfrak{t}}
\def\calC{{\mathcal C}}

\def\Z{\mathbb{Z}}

\def\K{\mathbb{K}}

\def\gl{{\mathfrak g\mathfrak l}}

\def\g{{\mathfrak{g}}}

\newcommand{\nc}{\newcommand}

\def\tH{\tilde{H}}
\nc{\op}[1]{\mathop{\mathchoice{\mbox{\rm #1}}{\mbox{\rm #1}}
{\mbox{\rm \scriptsize #1}}{\mbox{\rm \tiny #1}}}\nolimits}
\nc{\al}{\alpha}

\nc{\ep}{\varepsilon} \nc{\ga}{\gamma} \nc{\Ga}{\Gamma}
\nc{\la}{\lambda} \nc{\La}{\Lambda} \nc{\si}{\sigma}
\nc{\Sig}{{\Gamma}} \nc{\Om}{\Omega} \nc{\om}{\omega}

\nc{\SL}{{\rm SL}} \nc{\GL}{{\rm GL}} \nc{\PGL}{{\rm PGL}}
\nc{\G}{{\rm G}}

\def\U{{\mathcal{U}}}

\nc{\cpt}{{\op{cpt}}} \nc{\Dol}{{\op{Dol}}} \nc{\DR}{{\op{DR}}}
\nc{\B}{{\op{B}}} \nc{\Triv}{\op{Triv}} \nc{\Hod}{{\op{Hod}}}
\nc{\Log}{{\op{Log}}} \nc{\Exp}{{\op{Exp}}} \nc{\Est}{E_{\op{st}}}
\nc{\Hst}{H_{\op{st}}} \nc{\Left}[1]{\hbox{$\left#1\vbox to
  10.5pt{}\right.\nulldelimiterspace=0pt \mathsurround=0pt$}}
\nc{\Right}[1]{\hbox{$\left.\vbox to
  10.5pt{}\right#1\nulldelimiterspace=0pt \mathsurround=0pt$}}
\nc{\LEFT}[1]{\hbox{$\left#1\vbox to
  15.5pt{}\right.\nulldelimiterspace=0pt \mathsurround=0pt$}}
\nc{\RIGHT}[1]{\hbox{$\left.\vbox to
 15.5pt{}\right#1\nulldelimiterspace=0pt \mathsurround=0pt$}}

\nc{\bee}{{\bf E}} \nc{\bphi}{{\bf \Phi}}

\begin{document}

\title{Tensor products of unipotent characters of general linear groups over finite fields}

\author{Emmanuel Letellier \\ {\it Universit\'e de Caen} \\ {\tt letellier.emmanuel@math.unicaen.fr}}

\pagestyle{myheadings}

\maketitle

\begin{abstract} Given unipotent characters $\U_1,\dots,\U_k$ of $\GL_n(\F_q)$, we prove that  $\left\langle \U_1\otimes\cdots \otimes \U_k,1\right\rangle$ is a polynomial in $q$ with non-negative integer coefficients (this was  observed for  $n\leq 8$ and $k=3$ by Hiss-L\"ubeck-Mattig \cite{HLM}). We study the degree of this polynomial and give a necessary and sufficient condition in terms of the representation theory of symmetric groups and root systems for this polynomial to be non-zero.\end{abstract}

\tableofcontents

\section{The main results} 

Recall that the complex unipotent characters of $\GL_n(\F_q)$ are naturally parameterized by the irreducible characters of the symmetric group $\calS_n$ and therefore by the partitions of $n$. For a partition $\mu$ of $n$ we put $\U_\mu$ the corresponding unipotent character. Under our parametrization, the trivial character of $\GL_n$ is $\calU_{(n^1)}$ and the Steinberg character is $\calU_{(1^n)}$. 

Fix an integer $g\geq 0$ and consider $\calE:\GL_n(\F_q)\rightarrow\C$, $x\mapsto q^{g\,{\rm dim}\, C_{\GL_n(\overline{\F}_q)}(x)}$. If $g=1$, this is the character of the representation of $\GL_n(\F_q)$ in the group algebra $\C[\gl_n(\F_q)]$ where $\GL_n$ acts on $\gl_n$ by conjugation.

Using the inner product formula  

$$
\left\langle f,h \right\rangle=\left\langle f,h \right\rangle_{\GL_n(\F_q)}=\frac{1}{|\GL_n(\F_q)|}\sum_{x\in\GL_n}f(x)\overline{h(x)}
$$
which holds for any two class functions $f,h:\GL_n(\F_q)\rightarrow\C$, it is not difficult to see, using the character table of $\GL_n(\F_q)$ due to Green \cite{green}, that for any multi-partition $\muhat=(\mu^1,\dots,\mu^k)$ of $n$, there exists a polynomial $U_\muhat(t)\in\Q[t]$ such that for any finite field $\F_q$, we have 

$$
U_\muhat(q)=\left\langle\calE\otimes\U_{\mu^1}\otimes\cdots\otimes\U_{\mu^k},1\right\rangle.$$ 

The aim of this paper is to study the polynomials $U_\muhat(t)$. They were computed for $k=3$, $g=0$ and $1\leq n\leq 8$ by F. L\"ubeck \cite{Lubeck}. Some results of this paper can be easily observed in these tables.

\subsection{Generic case}In order to state our main theorem on the polynomials $U_\muhat(q)$ we need to introduce an other class of polynomials $V_\muhat(q)$. We say that a tuple $(\calX_1,\dots,\calX_k)$ of irreducible characters of $\GL_n(\F_q)$ is of \emph{type} $\muhat=(\mu^1,\dots,\mu^k)$ if for all $i=1,\dots,k$, there exists a linear character $\alpha_i:\F_q^\times\rightarrow\C^\times$ such that 

$$
\calX_i=(\alpha_i\circ {\rm det})\cdot\U_{\mu^i}.
$$
Such a tuple is said to be \emph{generic} if  the linear character $\alpha_1\alpha_2\cdots\alpha_k$ is of order $n$. 

In \cite[\S 6.10.6]{letellier4} (for a review see also \S \ref{generic}) we define  polynomials $V_\muhat(t)\in\Q[t]$  for any multi-partition $\muhat$ and prove that for any finite field $\F_q$ and any generic tuple $(\calX_1,\dots,\calX_k)$ of irreducible characters of $\GL_n(\F_q)$ of type $\muhat$ we have 

$$
V_\muhat(q)=\left\langle \calE\otimes\calX_1\otimes\cdots\otimes\calX_k,1\right\rangle.
$$
 From a multi-partition $\muhat=(\mu^1,\dots,\mu^k)$, we define a comet-shaped graph $\Gamma_\muhat$ together with a dimension vector $\v_\muhat$ as in \S \ref{quiver}. We then denote by $\Phi(\Gamma_\muhat)$ the associated root system as defined in \cite{kac}.

Put

$$
d_\muhat:=n^2(2g-2+k)-\sum_{i,j}(\mu^i_j)^2+2=2-{^t}\v_\muhat {\bf C}_\muhat\v_\muhat
$$
where ${\bf C}_\muhat$ is the Cartan matrix of $\Gamma_\muhat$ and $\mu^i=(\mu^i_1,\dots,\mu^i_{r_i})$ with $\mu^i_1\geq\mu^i_2\geq\cdots\geq\mu^i_{r_i}$.

\begin{theorem} (i) The polynomial $V_\muhat(t)$ is non-zero if and only if $\v_\muhat\in\Phi(\Gamma_\muhat)$. Moroever $V_\muhat(t)=1$ if and only if $\v_\muhat$ is a real root.

\noindent (ii) If non-zero, $V_\muhat(t)$ is a monic polynomial of degree $d_\muhat/2$ with non-negative integer coefficients.
\label{theo19}\end{theorem}

\begin{remark}In \cite{letellier4} we defined the notion of generic tuples $(\calX_1,\dots,\calX_k)$ for any types (not necessarily unipotent) of irreducible characters $\calX_1,\dots,\calX_k$ of $\GL_n(\F_q)$. Among these generic tuples, we defined a subclass whose elements are called \emph{admissible generic tuples}. We then proved that if $(\calX_1,\dots,\calX_k)$ is admissible  then the inner product $\langle\calE\otimes\calX_1\otimes\cdots\otimes\calX_k,1\rangle$ can be expressed as the Poincar\'e polynomial (for intersection cohomology) of a certain  quiver variety, from which we prove  a statement analogous to Theorem \ref{theo19}. Unfortunately, generic tuples of irreducible characters of unipotent type are never admissible and so we can not use the results of \cite{letellier4} to prove Theorem \ref{theo19}.

\end{remark}

\subsubsection{Connection with quiver varieties}\label{pos}

Consider a generic tuple  $(\calC_1,\dots,\calC_k)$ of regular semisimple adjoint orbits of $\gl_n(\C)$ (see proof of Theorem \ref{theogeneric} for the definition of generic tuples) and
consider the space $\calV$ of tuples of matrices 

$$(A_1,\dots,A_g,B_1,\dots,B_g,X_1,\dots,X_k)\in\gl_n(\C)^{2g}\times\calC_1\times\cdots\times\calC_k$$ which satisfy the equation

$$[A_1,B_1]+\cdots+[A_g,B_g]+X_1+\cdots+X_k=0.
$$

Put 

$$
\calQ:=\calV/\!/\GL_n={\rm Spec}\,\left(\C[\calV]^{\GL_n}\right)$$ where $\GL_n$ acts diagonally by conjugation on $\calV$. The variety $\calQ$ is non-singular and the quotient map $\calV\rightarrow\calQ$ is a principal $\PGL_n$-bundle in the \'etale topology. Denote by $H_c^i(\calQ,\C)$ the compactly supported cohomology of $\calQ$. Recall (see for instance \cite{hausel-letellier-villegas}) that $H_c^i(\calQ,\C)=0$ when $i$ is odd.

 We can define an action $\rho^i$ of $k$ copies $\mathbb{S}_n:=\mathfrak{S}_n\times\cdots\times \mathfrak{S}_n$ of the symmetric group $\mathfrak{S}_n$ on $H_c^{2i}(\calQ,\C)$. This is a particular case of Weyl group actions  on cohomology of quiver varieties constructed and studied by many authors including Nakajima \cite{nakajima} \cite{nakajima2}, Lusztig \cite{Lusztig}, Maffei \cite{maffei}. The construction of the Weyl group action given  in \cite{letellier4} does not apply here (we can only construct the  action of some relative Weyl groups which are finite subgroups of $\mathbb{S}_n$).

For a partition $\lambda$, denote by $\chi^\lambda$ the irreducible character of the symmetric group $\mathfrak{S}_n$ associated with $\lambda$ as in \cite{macdonald}. Following the strategy of \cite{hausel-letellier-villegas3} (see  proof of Theorem \ref{theogeneric} for more details) we can  show that for any multipartition $\muhat=(\mu^1,\dots,\mu^k)$ of $n$ we have

\begin{equation}
V_\muhat(t)=q^{-d/2}\sum_i\left\langle\chi^{\muhat'},\rho^i\right\rangle_{\mathbb{S}_n}t^i
\label{posfor}\end{equation}
where $d$ is the dimension of $\calQ$, $\muhat'$ denotes the dual multi-partition of $\muhat$ and $\chi^\muhat$ is the irreducible character $\chi^{\mu^1}\otimes\cdots\otimes\chi^{\mu^k}$ of $\mathbb{S}_n$. Formula (\ref{posfor}) implies the positivity of the coefficients of $V_\muhat(t)$. Theorem \ref{theo19}(i) together with  Formula (\ref{posfor}) provides a nice criterion in terms of roots for the appearance or not  of an irreducible character of $\mathbb{S}_n$ in $\rho^*:=\bigoplus_i\rho^i$.

\subsubsection{Connection with character varieties}\label{CV}

Let us recall the conjectural interpretation of the polynomials $V_\muhat(t)$ in terms of Poincar\'e polynomial of character varieties \cite[\S 1.3]{letellier4}.

For a partition $\lambda$ of $n$ let us denote by $C_\lambda$ the unipotent conjugacy class of $\GL_n(\C)$ whose size of Jordan blocks is given by the dual partition $\lambda'$ of $\lambda$. 

For a multi-partition $\muhat=(\mu^1,\dots,\mu^k)$ of $n$, put 

$$
\overline{C}_\muhat:=\GL_n^{2g}\times \overline{C}_{\mu^1}\times\cdots\times \overline{C}_{\mu^k}.
$$
Fix primitive $n$-th root of unity $\zeta$ and consider the space $\calZ_\muhat$ of tuples

$$
(A_1,\dots,A_g,B_1,\dots,B_g,X_1,\dots,X_k)\in \overline{C}_\muhat
$$
which satisfy the equation

$$
\prod_{i=1}^g(A_i,B_i)\prod_{j=1}^kX_i=\zeta\cdot I_n
$$
where $I_n$ is the identity matrix and $(A,B)$ is the commutator $ABA^{-1}B^{-1}$. Put

$$
\M_\muhat:=\calZ_\muhat/\!/\GL_n={\rm Spec}\,\left(\C[\calZ_\muhat]^{\GL_n}\right)
$$
where $\GL_n$ acts diagonally by conjugation on $\calZ_\muhat$.  By Saito \cite{saito} the compactly supported intersection cohomology $IH_c^i(\M_\muhat,\C)$ is endowed with a mixed Hodge structure. Denote by $\{ih_c^{r,s;k}(\M_\muhat)\}_{r,s,k}$ the corresponding  mixed Hodge numbers and consider the pure part 
$$
PP_c(\M_\muhat,t):=\sum_s ih_c^{s,s;2s}(\M_\muhat)t^s
$$
of the mixed Poincar\'e polynomial.
Then we have the following conjecture \cite[Conjecture 1.3.2]{letellier4}  

\begin{conjecture}$$
V_\muhat(t)=t^{-d_\muhat/2}PP_c(\M_\muhat,t).
$$
\end{conjecture}

\subsection{Unipotent case}

In order to see the relation between the two polynomials $U_\muhat(t)$ and $V_\muhat(t)$, we need to introduce some notations. Consider $k$ separate sets $\x_1,\x_2,\dots,\x_k$ of infinitely many variables and denote by $\Lambda(\x_1,\dots,\x_k)=\Lambda(\x_1)\otimes_\Z\cdots\otimes_\Z\Lambda(\x_k)$ the ring of  functions separately symmetric in each set $\x_1,\dots,\x_k$, and put $\Lambda=\Q(t)\otimes_\Z\Lambda(\x_1,\dots,\x_k)$. For a multi-partition $\muhat=(\mu^1,\dots,\mu^k)$, we define  $s_\muhat\in\Lambda$ by 
$$
s_\muhat:=s_{\mu^1}(\x_1)\cdots s_{\mu^k}(\x_k)
$$
where for a partition $\lambda$ we denote by $s_\lambda(\x_i)\in\Lambda(\x_i)$ the corresponding Schur symmetric function as  in \cite{macdonald}.

Denote by $\P$ the set of all partitions including the unique partition $0$ of $0$ and denote by  $\oP$ the set of multi-partitions $\muhat=(\mu^1,\dots,\mu^k)\in\P^k$ with $|\mu^1|=|\mu^2|=\cdots=|\mu^k|=:|\muhat|$. We denote by $\P_n$ and $\oP_n$ the subsets of partitions of size $n$. 

We prove the following result (see Proposition \ref{expgen}).

\begin{proposition} We have 

$$
\Exp\left(\sum_{\muhat\in\oP-\{0\}} V_\muhat(t) s_\muhat T^{|\muhat|}\right)=1+\sum_{\muhat\in\oP-\{0\}} U_\muhat(t) s_\muhat T^{|\muhat|}
$$
where $\Exp: T\Lambda[[T]]\rightarrow 1+T\Lambda[[T]]$ is the plethystic exponential.
\end{proposition}

Our strategy to study the polynomials $U_\muhat(t)$ is to use the above proposition together with the properties of $V_\muhat(t)$.

Consider now a total ordering $\geq$ on the set of all partitions. Denote by $\bT_n^o$ the set of non-increasing sequences of partitions $\alpha^1\alpha^2\cdots\alpha^r$ such that $\sum_{i=1}^r|\alpha^i|=n$. We will write the elements $\omega^o\in\bT_n^o$ in the form $(\alpha^1)^{n_1}(\alpha^2)^{n_2}\cdots (\alpha^s)^{n_s}$ with $\alpha^1>\alpha^2>\cdots>\alpha^s$ and with $n_i$ the multiplicity of $\alpha^i$ in $\omega^o$. We then put

$$
\calS_{\omega^o}:=\prod_{i=1}^s(\calS_{|\alpha^i|})^{n_i},\hspace{.5cm}H_{\omega^o}:=\bigotimes_{i=1}^sT^{n_i}H_{\alpha^i}
$$
where for a partition $\lambda$, $H_\lambda$ is an irreducible $\C[\calS_{|\lambda|}]$-module with character $\chi^\lambda$, and where $T^mV:=V\otimes\cdots \otimes V$, with $V$ repeated $m$ times. For a partition $\mu$ of $n$ and a type $\omega^o\in\bT_n^o$ define

$$
\C_{\omega^o}^\mu:={\rm Hom}_{\calS_n}\left({\rm Ind}_{\calS_{\omega^o}}^{\calS_n}(H_{\omega^o}),H_\mu\right)
$$
where for an inclusion of finite groups $H\subset K$, we denote by ${\rm Ind}_H^K$ is the usual induction functor $V\mapsto \C[K]\otimes_{\C[H]}V$ from the category of left $\C[H]$-modules into the category of left $\C[K]$-modules. In \cite[\S 6]{letellier4} we constructed an action of the group $W_{\omega^o}:=\calS_{n_1}\times\cdots \times \calS_{n_s}$ on the $\C$-vector space $\C_{\omega^o}^\mu$ and we  proved that given a partition $\nu_i=(d_{i,1},\dots,d_{i,r_i})$ of $n_i$ for all $i=1,\dots,s$, the coordinates of 

$$
s_{\alpha^1}(\x^{d_{1,1}})\cdots s_{\alpha^1}(\x^{d_{1,r_1}})s_{\alpha^2}(\x^{d_{2,1}})\cdots s_{\alpha^2}(\x^{d_{2,r_2}})\cdots s_{\alpha^s}(\x^{d_{s,1}})\cdots s_{\alpha^s}(\x^{d_{s,r_s}})
$$in the basis of Schur symmetric functions $\{s_\mu\}_\mu$ 
equal ${\rm Trace}\,\left(w\,|\, \C_{\omega^o}^\mu\right)$ where $w=(w_1,\dots,w_s)\in W_{\omega^o}$ with $w_i$ in the conjugacy class of $\calS_{n_i}$ corresponding to the partition $\nu_i$.
 Choose once for all  a total ordering on the set of multi-partitions $\oP$ and denote by $\oT{^o_n}$ the set of non-increasing sequences $\alphahat_1^{n_1}\alphahat_2^{n_2}\cdots\alphahat_s^{n_s}$ such that $\alphahat_1>\alphahat_2>\cdots>\alphahat_s$ and 
 
 $$
 \sum_{i=1}^sn_i|\alphahat_i|=n.
 $$
It will be also convenient in this paper to think of the element  $\omhat^o=\alphahat_1^{n_1}\alphahat_2^{n_2}\cdots\alphahat_s^{n_s}\in\oT{^o_n}$ as a function $\omhat^o:\oP\rightarrow\N$ with $\omhat^o(\alphahat_i)=n_i$ and $\omhat^o(\muhat)=0$ if $\muhat\notin\{\alphahat_1,\dots,\alphahat_s\}$.

The total orderings on $\P$ and $\oP$ defines  a natural map $\oT{^o_n}\rightarrow(\bT_n^o)^k$, $\omhat^o\mapsto (\omega^o_1,\dots,\omega^o_k)$. For $\muhat=(\mu^1,\dots,\mu^k)\in\oP_n$ and $\omhat^o\in\bT{^o_n}$, define 

$$
\C_{\omhat^o}^\muhat:=\bigotimes_{i=1}^k\C_{\omega^o_i}^{\mu^i}.
$$
If $\omhat^o=\alphahat_1^{n_1}\alphahat_2^{n_2}\cdots\alphahat_s^{n_s}$, the group $W_{\omhat^o}:=\prod_{i=1}^s\calS_{n_i}$ acts on $\C_{\omhat^o_i}^{\muhat}$ via its diagonal embedding in $W_{\omega^o_1}\times\cdots\times W_{\omega^o_k}$. We then define

$$
\calR_{\omhat^o,\,\muhat}:=\left\{(\tau^1,\dots,\tau^s)\in\calP_{n_1}\times\cdots\times\calP_{n_s}\,\left|\,\left\langle H_{\tau^1}\otimes\cdots\otimes H_{\tau^s},\C_{\omhat^o}^\muhat\right\rangle_{W_{\omhat^o}}\neq 0\right.\right\}.
$$

For a partition $\lambda$, we denote by $\ell(\lambda)$ its length. We can now state the main result of this paper.

\begin{theorem} Let $\muhat\in\oP_n$ with $n\geq 1$.

\noindent (i) The polynomial $U_\muhat(t)$ has non-negative integer coefficients.

\noindent (ii) The polynomial $U_\muhat(t)$ is non-zero if and only if there exists $\omhat^o=\alphahat_1^{n_1}\alphahat_2^{n_2}\cdots\alphahat_s^{n_s}\in\oT{^o_n}$ and $(\tau^1,\dots,\tau^s)\in\calR_{\omhat^o,\,\muhat}$ such that 

\begin{equation}
\ell(\tau^i)\leq V_{\alphahat_i}(1)
\label{1}\end{equation} for all $i=1,\dots,s$.
\end{theorem}

By Theorem \ref{theo19}, the inequality (\ref{1}) does not hold unless $\v_{\alphahat_i}$ is a root of $\Gamma_{\alphahat_i}$. Denote by $\oT{^o_{n\,+}}$ the subset of $\oT{^o_n}$ of sequences $\alphahat_1^{n_1}\alphahat_2^{n_2}\cdots\alphahat_s^{n_s}$ with $\v_{\alphahat_i}\in\Phi(\Gamma_{\alphahat_i})$.

\begin{corollary}Let $\muhat\in\oP_n$ with $n\geq 1$. If there exists $\omhat^o\in\oT{^o_{n\,+}}$ such that $\left\langle \C_{\omhat^o}^\muhat,1\right\rangle\neq 0$, then $U_\muhat(t)\neq 0$. 
\label{coro1}\end{corollary}

Let $\omhat^o=\alphahat_1^{n_1}\alphahat_2^{n_2}\cdots\alphahat_s^{n_s}\in\oT{^o_n}$ and $\muhat=(\mu^1,\dots,\mu^k)\in\oP_n$. 

Notice that if $\alphahat_i=((1),\dots,(1))$ for all $i=1,\dots,s$, then $W_{\omhat^o}=\calS_n$ and $\C_{\omhat^o}^\muhat=H_\muhat$. 

Note also that if $n_1=n_2=\cdots=n_s$, then $W_{\omhat^o}=1$ and so 

$$
\left\langle \C_{\omhat^o}^\muhat,1\right\rangle={\rm dim}\, \C_{\omhat^o}^\muhat
$$
is a product of Littlewood-Richardson coefficients. If moreover $s=1$ and $n_1=1$, then $\left\langle\C_{\omhat^o}^\muhat,1\right\rangle=\delta_{\omhat^o,\, \muhat}$.

In particular it follows from Corollary \ref{coro1} that if  $\v_\muhat\in\Phi(\Gamma_\muhat)$ or if $\langle H_\muhat,1\rangle\neq 0$, then $U_\muhat(t)\neq 0$.

We can actually prove the following result (see Proposition \ref{suffcond} and Remark \ref{cases}).

\begin{proposition}The term $\langle H_\muhat,1\rangle$ contributes to the constant term of $U_\muhat(t)$.
\end{proposition}

For $\muhat=(\mu^1,\dots,\mu^k)$ put

$$
\delta(\muhat)=(2g-2+k)n-\sum_{i=1}^k\mu^i_1.
$$

Then $\v_\muhat$ is in the fundamental set of imaginary roots of $\Gamma_\muhat$ if and only if $\delta(\muhat)\geq 0$. Note also that if $g\geq 1$, then $\delta(\muhat)\geq 0$ and so in this case $U_\muhat(t)$ is always non-zero.

We prove the following theorem concerning the degree of $U_\muhat(t)$.

\begin{theorem}(i) If $\v_\muhat\in\Phi(\Gamma_\muhat)$, then the degree of $U_\muhat(t)$ is at least $d_\muhat/2$.

\noindent (ii) If $\delta(\muhat)\geq 2$, then the degree of $U_\muhat(t)$ is exactly $d_\muhat/2$.

\noindent (iii) If $\delta(\muhat)\geq 3$ or $g=0$, $k=3$ and $\delta(\muhat)=2$, then $U_\muhat(t)$ is monic.
\label{deg}\end{theorem}

This theorem can be used to reduce the proof of results of the following kind to a finite number of checks (see proof of Corollary \ref{stein}).

\begin{corollary} Let ${\rm St}_n$ denotes the Steinberg character of $\GL_n(\F_q)$. Then for all $n\geq 1$, the inner product $\langle {\rm St}_n\otimes{\rm St}_n\otimes{\rm St}_n,1\rangle$ is a monic polynomial in $q$ of degree $\frac{1}{2}(n-1)(n-2)$.
\end{corollary}

\section{Preliminaries}

\subsection{Log and Exp}\label{Log}

Fix an integer $k>0$. Consider $k$ separate sets $\x_1,\x_2,\dots,\x_k$ of infinitely many variables and denote  by   $\Lambda(\x_1,\dots,\x_k):=\Lambda(\x_1)\otimes_\Z\cdots\otimes_\Z\Lambda(\x_k)$ the ring of functions separately symmetric in each set $\x_1,\dots,\x_k$. Put $\Lambda:=\Q(t)\otimes_\Z\Lambda(\x_1,\dots,\x_k)$. 

Consider  $$\psi_n:\Lambda[[T]]\rightarrow\Lambda[[T]],\, f(\x_1,\dots,\x_k;t,T)\mapsto f(\x_1^n,\dots,\x_k^n;t^n,T^n)$$where we denote by $\x^d$ the set of variables $\{x_1^d,x_2^d,\dots\}$. The $\psi_n$ are called the \emph{Adams operations}.

Define $\Psi:T\Lambda[[T]]\rightarrow T\Lambda[[T]]$ by $$\Psi(f)=\sum_{n\geq 1}\frac{\psi_n(f)}{n}.$$Its inverse is given by $$\Psi^{-1}(f)=\sum_{n\geq 1}\mu(n)\frac{\psi_n(f)}{n}$$where $\mu$ is the ordinary M\"obius function. 

Following Getzler \cite{getzler} we define $\Log:1+T\Lambda[[T]]\rightarrow T\Lambda[[T]]$ and its inverse $\Exp:T\Lambda[[T]]\rightarrow 1+T\Lambda[[T]]$ as 

$$\Log(f)=\Psi^{-1}\left(\log(f)\right)$$and $$\Exp(f)=\exp\left(\Psi(f)\right).$$

\begin{lemma} Let $f\in T\Lambda[[T]]$. If $f$ has coefficients in $\Z[t]\otimes_\Z\Lambda(\x_1,\dots,\x_k)\subset\Lambda$, then $\Exp(f)$ has also coefficients in  $\Z[t]\otimes_\Z\Lambda(\x_1,\dots,\x_k)$.
\label{Moz}
\end{lemma}

\begin{proof} We could have defined  $\Exp$ using the $\sigma$-operations instead of the $\psi$-operations in which case the above lemma becomes clear, see for instance \cite{mozgovoy} for more details.
\end{proof}

For $g\in \Lambda$ and $n\geq 1$ we put 

$$
g_n:=\frac{1}{n}\sum_{d | n}\mu(d)\psi_{\frac{n}{d}}(g).
$$
This is the M\"obius inversion formula of $\psi_n(g)=\sum_{d| n}d\cdot g_d$.

We have the following lemma \cite{mozgovoy}.

\begin{lemma} Let $g\in \Lambda$ and $f_1,f_2\in 1+T\Lambda[[T]]$ such that

$$
\log\,(f_1)=\sum_{d=1}^\infty g_d\cdot\log\,(\psi_d(f_2)).
$$
Then 
$$
\Log\,(f_1)=g\cdot \Log\,(f_2).
$$
\label{moz}\end{lemma}

\subsection{Partitions, types}\label{parttype}

Denote by $\calP$ the set of all partitions including the unique partition $0$ of $0$ and by $\calP_n$ the subset of partitions of $n$.  Partitions $\lambda$ are denoted by $(\lambda_1,\lambda_2,\dots,\lambda_r)$ with $\lambda_1\geq\lambda_2\geq\cdots\geq\lambda_r\geq 0$. We will sometimes write a  partition $\lambda$ as $(1^{m_1},2^{m_2},\dots,i^{m_i})$ where $m_i$ denotes the multiplicity of $i$ in $\lambda$. The \emph{size} of $\lambda$ is defined as $|\lambda|:=\sum_i\lambda_i$. If $d>0$ is an integer and $\lambda=(\lambda_1,\dots,\lambda_r)$ a partition of $n$, then  $d\cdot\lambda:=(d\lambda_1,\dots,d\lambda_r)\in\calP$ is a partition of $dn$. We also define the sum $\lambda+\mu$ of two partitions $\lambda=(\lambda_1,\dots,\lambda_r)$ and $\mu=(\mu_1,\dots,\mu_s)$ as the partition $(\lambda_1+\mu_1,\lambda_2+\mu_2,\dots)$. We consider on $\calP_n$ the partial ordering $\unlhd$ defined as follows. We have $\lambda\unlhd\mu$ if for all $i$,  $\lambda_1+\cdots+\lambda_i\leq\mu_1+\cdots+\mu_i$. We denote by $\oP$ the set of multi-partitions $\muhat=(\mu^1,\mu^2,\dots,\mu^k)\in\calP^k$ such that $|\mu^1|=|\mu^2|=\cdots=|\mu^k|$ and we extend in the obvious way the definitions of $d\cdot\muhat$ (with $d\in\N$) and $\lambdahat+\muhat$. We denote by $\oP_n$ the subset of multi-partitions in $\oP$ of size $n$. Finally we say that $\lambdahat\unlhd\,\muhat$ if and only if $\lambda^i\unlhd\,\mu^i$ for all $i=1,\dots,k$.

We call  \emph{multi-type} a function $\omhat: \N\times\oP\rightarrow \N$ such that its  support $S_\omhat:=\{(d,\muhat)\,|\,\omhat(d,\muhat)\neq 0\}$ is finite and does not contain pairs of the form $(0,\muhat)$ or $(d,0)$. We denote by $0$ the multi-type corresponding to the zero function. The \emph{degrees} of a multi-type $\omhat$ are the integers $d$ such that $(d,\muhat)\in S_\omhat$ for some $\muhat\in\oP$. If the degres of $\omhat$ are all equal to $1$, we say that $\omhat$ is \emph{split}. We call $|\omhat|:=\sum_{(d,\muhat)\in S_\omhat}d\cdot|\muhat|\cdot\omhat(d,\muhat)$ the \emph{size} of $\omhat$.  We denote by $\oT$ the set of all multi-types as above and by $\oT_n$ the subset of multi-types of size $n$. We use the notation $\bT$ (resp. $\bT_n$) instead of $\oT$ (resp. $\oT_n$) for $k=1$, and call simply an element of $\bT$ a \emph{type}.

Note that we have a natural map $\oT\rightarrow \bT^k$ as follows. If $\omhat\in \oT$, then for each $i=1,\dots,k$, we define its $i$-th coordinate  $\omega_i:\N\times \calP\rightarrow \N$ as $\omega_i(d,\mu)=\sum_\muhat\omhat(d,\muhat)$ where the sum is over the elements $\muhat\in \oP$ whose $i$-th coordinate is $\mu$.

Finally for $\omhat\in\oT$ we define the multi-partition $\omhat_+\in\oP$ as 

$$
\omhat_+:=\sum_{(d,\,\muhat)\in S_\omhat}(d\,\omhat(d,\muhat))\cdot\muhat.
$$

Given a family $\{a_\muhat\}_{\muhat\in\oP}$ of elements of $\Lambda$, we extend its definition to multi-types $\omhat\in\oT$ as 

$$
a_\omhat:=\prod_{(d,\muhat)\in S_\omhat}\psi_d(a_\muhat)^{\omhat(d,\muhat)}.
$$

For a multi-type $\omhat\in\oT-\{0\}$, define

$$C_\omhat^o:=\begin{cases}\frac{\mu(d)}{d}(-1)^{r_\omhat-1}\frac{(r_\omhat-1)!}{\prod_\muhat \omhat(d,\muhat)!}\,\text{ if there is no } (d',\muhat)\in S_\omhat \, \text{ with } d'\neq d.\\0\,\text{ otherwise.}
\end{cases}
$$
where $\mu$ is the ordinary M\"obius function and $r_\omhat:=\sum_{(d,\muhat)\in S_\omhat}\omhat(d,\muhat)$.

We have the following lemma \cite[\S 2.3.3]{hausel-letellier-villegas}.

\begin{lemma} 
Let $\{a_\muhat\}_{\muhat\in\oP}$ be a family of elements of  $\Lambda$ with $a_0=1$. Then
\begin{equation}
\Log\left(\sum_{\muhat\in\oP}a_\muhat T^{|\muhat|}\right)=\sum_{\omhat\in\oT-\{0\}}C_\omhat^oa_\omhat T^{|\omhat|}.
\end{equation}\label{Log-w}
\end{lemma}

For a multi-type $\omhat\in\oT$, define

$$
A_\omhat^o:=\prod_{(r,\muhat)\in S_\omhat}\frac{1}{r^{\omhat(r,\muhat)}\omhat(r,\muhat)!}.
$$

The following lemma is also straightforward.

\begin{lemma}Let $\{a_\muhat\}_{\muhat\in\oP}$ be a family of elements of  $\Lambda$ with $a_0=1$. Then 

$$
\Exp\left(\sum_{\muhat\in\oP-\{0\}}a_\muhat T^{|\muhat|}\right)=\sum_{\omhat\in\oT}A_\omhat^oa_\omhat T^{|\omhat|}.
$$
\end{lemma}

The formal power series $\sum_{n\geq 0}a_nT^n$ with $a_n\in \Lambda$
that we will consider in what follows will all have $a_n$ homogeneous
of degree $n$ in the variables $\x_1,\dots,\x_k$. Hence we will typically  scale the variables of
$\Lambda$ by $1/T$ and eliminate $T$ altogether.

\subsection{Littlewood-Richardson coefficients}\label{LR}

For a partition $\lambda\in\calP$ we denote by $s_\lambda(\x)\in\Lambda(\x)$ the corresponding Schur function. For a type $\omega\in\bT$ and a partition $\mu\in\calP$, define $c_\omega^\mu\in\Z$ by

$$
s_\omega=\sum_{\mu\unlhd\,\omega_+}c_\omega^\mu s_\mu.
$$
Note that $c_\omega^\mu=0$ unless $|\omega|=|\mu|$. If $\omega$ is split, then $c_\omega^\mu$ is a so-called \emph{Littlewood-Richardson coefficient}.

For an integer $n>0$, we denote by $\calS_n$ the symmetric group in $n$ letters.

For a finite dimensional $\C$-vector space $V$ and an integer $d>0$, we put $T^dV:=V\otimes \cdots\otimes V$ with $V$ repeated $d$ times.

For a partition $\lambda$, we denote by $H_\lambda$ an irreducible $\C[\calS_{|\lambda|}]$-module corresponding to the irreducible character $\chi^\lambda$ of $\calS_{|\lambda|}$. Here we use the same parametrization $\lambda\mapsto \chi^\lambda$ as in  \cite{macdonald}; the trivial character of $\calS_n$ corresponds to the partition $(n)$.

Define $\tT$ as the set of functions $\tomhat:\calP\rightarrow \N$ whose support $S_{\tomega}:=\{\mu\,|\,\tomega(\mu)\neq 0\}$ is finite and does not contain the $0$ element of $\calP$. 

Note that we have a natural map $\mathfrak{H}:\bT\rightarrow\tT$ that maps $\omega$ to the function $\tomega$ defined by $\tomega(\mu)=\sum_dd\cdot \omega(d,\mu)$.

Given a type $\tomega\in\tT$, we put

$$
H_\tomega:=\bigotimes_{\mu\in S_\tomhat}T^{\tomega(\mu)}H_\mu, \hspace{1cm}\calS_\tomega:=\prod_{\mu\in S_\tomega}(\calS_{|\mu|})^{\tomega(\mu)}, \hspace{1cm}W_\tomega:=\prod_{\mu\in S_\tomega}\calS_{\tomega(\mu)}.
$$
The elements of the fiber $\mathfrak{H}{^{-1}}(\tomhat)$ are in bijection with $\prod_{\mu\in S_\tomega}\calP_{\tomega(\mu)}$ and so with the conjugacy classes of $W_\tomega$.  For $\tomega\in\tT$ and $\mu\in\calP$ we define 

$$
\C_\tomega^\mu:={\rm Hom}_{\calS_{|\tomega|}}\left({\rm Ind}_{\calS_\tomega}^{\calS_{|\tomhat|}}(H_\tomega),H_\mu\right).
$$

 The normalizer $N_{\calS_{|\tomhat|}}(\calS_\tomega)$ of $\calS_\tomega$ in $\calS_{|\tomega|}$ acts on the set of representations of $\calS_\tomega$ on the left  as $\sigma\cdot\rho:=\rho\circ\sigma^{-1}$. Then we have an isomorphism

$$
W_\tomega\simeq \left.\left\{\sigma\in N_{\calS_{|\tomega|}}(\calS_\tomega)\,\right|\,\sigma\cdot\rho_\tomega\simeq\rho_\tomega\right\}/\calS_\tomega.
$$

By \cite[\S 6.2]{letellier4}, the group $W_\tomega$ acts on the space $\C_\tomega^\mu$ and we have the following proposition.

\begin{proposition} For all $v\in W_\tomega$ we have 

$$
{\rm Tr}\,\left(v\,\left|\,\C_\tomega^\mu\right)\right.=c_\omega^\mu
$$
where $\omega\in\bT$ is the element in the fiber $\mathfrak{H}{^{-1}}(\tomega)$ which corresponds to the conjugacy class of $v$. 
\label{let}\end{proposition}

We now extend this proposition to the case of multi-partitions.

Given any family $\{a_\mu\}$ of symmetric functions
indexed by partitions $\mu\in \P$ and a multi-partition  $\muhat=(\mu^1,\dots,\mu^k) \in
\oP$ define
$$
a_\muhat:=a_{\mu^1}(\x_1)\cdots a_{\mu^k}(\x_k).
$$

For a multi-type $\omhat\in\oT$ and a multi-partition $\muhat\in\oP$, we denote by $c_\omhat^\muhat$ the integer defined by 

$$
s_\omhat=\sum_{\muhat\unlhd\,\omhat_+} c_\omhat^\muhat s_\muhat.
$$

\begin{remark}Denote by $(\omega_1,\dots,\omega_k)$ the coordinates of $\omhat$ in $\bT^k$. Then 

$$
s_\omhat=s_{\omega_1}(\x_1)\cdots s_{\omega_k}(\x_k),
$$
and so we see that the coefficient $c_\omhat^\muhat$ is a product $c_{\omega_1}^{\mu^1}\cdots c_{\omega_k}^{\mu^k}$ of coefficients $c_{\omega_i}^{\mu^i}$ defined above where $\mu^1,\dots,\mu^k$ are the coordinates of $\muhat$.\label{rem1}\end{remark}

Define $\otT$ as the set of functions $\tomhat:\oP\rightarrow \N$ whose support $S_{\tomhat}:=\{\muhat\,|\,\tomhat(\muhat)\neq 0\}$ is finite and does not contain the $0$ element of $\oP$. If $k=1$, then $\otT$ is simply $\tT$ defined above. We then have a natural map $\otT\rightarrow (\tT)^k$ that sends $\tomhat$ to $(\omega_1^o,\dots,\omega_k^o)$ with $\omega_i^o(\mu):=\sum_\muhat\tomhat(\muhat)$ where the sum is over the elements $\muhat\in\oP$ whose $i$-th coordinate is $\mu$. We call $\omega_i^o$ the $i$-th coordinate of $\tomhat$.

Given a multi-type $\tomhat\in\otT$, we put

$$
W_\tomhat:=\prod_{\muhat\in S_\tomhat}\calS_{\tomhat(\muhat)}.
$$

For all $i=1,\dots,k$, the group $W_\tomhat$ is a subgroup of $W_{\omega_i^o}$.

Consider the map $\overline{\mathfrak{H}}:\oT\rightarrow\otT$ that maps $\omhat$ to the function $\tomhat$ defined by $\tomhat(\muhat)=\sum_dd\cdot \omhat(d,\muhat)$. The elements of the fiber $\overline{\mathfrak{H}}{^{-1}}(\tomhat)$ are then in bijection with $\prod_{\muhat\in S_\tomhat}\calP_{\tomhat(\muhat)}$ and so with the conjugacy classes of $W_\tomhat$.

For $\omhat^o\in\otT$ with coordinates $(\omega_1^o,\dots,\omega_k^o)$ and $\muhat=(\mu^1,\dots,\mu^k)\in\oP$ we define 

$$
\C_\tomhat^\muhat:=\bigotimes_{i=1}^{k}\C_{\omega_i^o}^{\mu^i}
$$

The group $W_{\omega_1^o}\times\cdots \times W_{\omega_k^o}$ acts on $\C_\tomhat^\muhat$ and  so does the group  $W_\tomhat$  via its diagonal embedding in  $W_{\omega_1^o}\times\cdots\times W_{\omega_k^o}$.

The following proposition is a consequence of Proposition \ref{let} and Remark \ref{rem1}.

\begin{proposition} For all $v\in W_\tomhat$ we have 

$$
{\rm Tr}\,\left(v\,\left|\,\C_\tomhat^\muhat\right)\right.=c_\omhat^\muhat
$$
where $\omhat\in\oT$ is the element in the fiber $\overline{\mathfrak{H}}{^{-1}}(\tomhat)$ which corresponds to the conjugacy class of $v$. 

\label{LRb}\end{proposition}

\subsection{A technical result}\label{tech}

Assume given a family  $\{V_\muhat(t)\}_{\muhat\in\oP}$ of polynomials in $\Z[t]$ indexed by $\oP$. Let $\{U_\muhat(t)\}_{\muhat\in\oP-\{0\}}$ be the family defined by $$\Exp\left(\sum_{\muhat\in\oP-\{0\}}V_\muhat(t)s_\muhat\right)=1+\sum_{\muhat\in\oP-\{0\}} U_\muhat(t)s_\muhat.$$By Lemma \ref{Moz} we have $U_\muhat(t)\in\Z[t]$.

The aim of this section is to study the properties of the polynomials $U_\muhat(t)\in\Z[t]$ in terms of those of $V_\muhat(t)$.

 We have 

\begin{align*}
U_\muhat(t)&=\left\langle\sum_{\omhat\in\oT}A_\omhat^oV_\omhat(t) s_\omhat,s_\muhat\right\rangle\\
&=\sum_{\omhat\in\oT}A_\omhat^oV_\omhat(t)c_\omhat^\muhat\\
&=\sum_{\tomhat\in\otT}\sum_{\omhat\in \overline{\mathfrak{H}}{^{-1}}(\tomhat)}A_\omhat^oV_\omhat(t)c_\omhat^\muhat\\
&=\sum_{\tomhat\in\otT}W_\muhat^\tomhat(t)
\end{align*}
where
$$
W_\muhat^\tomhat(t):=\sum_{\omhat\in \overline{\mathfrak{H}}{^{-1}}(\tomhat)}A_\omhat^oV_\omhat(t)c_\omhat^\muhat.
$$
Now fix $\tomhat\in\otT$. Write  $S_\tomhat=\{\alphahat_1,\dots,\alphahat_s\}$ and put $n_i:=\tomhat(\alphahat_i)$.

Recall that the elements of $\overline{\mathfrak{H}}{^{-1}}(\tomhat)$ are naturally parameterized by the set $\calP_{n_1}\times\cdots\times\,\calP_{n_s}$. If $\omhat\in\overline{\mathfrak{H}}{^{-1}}(\tomhat)$ corresponds to $(\lambda^1,\dots,\lambda^s)\in\calP_{n_1}\times\cdots\times\,\calP_{n_s}$, then $A_\omhat^o=z_{\lambda^1}\cdots z_{\lambda^s}$ where for a partition $\lambda=(1^{m_1},2^{m_2},\dots)$ we put 

$$
z_\lambda:=\prod_{i\geq 1}i^{m_i}\cdot m_i!.
$$
Recall that $z_\lambda$ is the cardinality of the centralizer in $\calS_{|\lambda|}$ of an element of type $\lambda$. For a partition $\lambda$, denote by $p_\lambda(\x)\in\Lambda(\x)$ the corresponding power symmetric function in the infinite set of variables $\x=\{x_1,x_2,\dots\}$. Then $p_1(\x)=x_1+x_2+\cdots$. Let $\y_1,\dots,\y_s$ be $s$ independent sets of infinitely many variables.

Consider

$$
W_\muhat^\tomhat(\y_1,\dots,\y_k):=\sum_{(\lambda^1,\dots,\lambda^s)\in\calP_{n_1}\times\cdots\times\,\calP_{n_s}}\frac{1}{z_{\lambda^1}\cdots z_{\lambda^s}}\,p_{\lambda^1}(\y_1)\cdots p_{\lambda^s}(\y_s)\,{\rm Tr}\,\left(v_{(\lambda^1,\dots,\lambda^s)}\,\left|\, \C_\tomhat^\muhat\right)\right.$$where $v_{(\lambda^1,\dots,\lambda^s)}\in W_\tomhat$ is a representative of the conjugacy class of $W_\tomhat$ corresponding to $(\lambda^1,\dots,\lambda^s)$. 

\begin{lemma}Assume that for all $i=1,\dots,s$,  the polynomial $V_{\alphahat_i}(t)$ has non-negative integer coefficients, then for an appropriate specialization of the variables $\y_1,\dots,\y_s$, we have 

$$
W_\muhat^\tomhat(\y_1,\dots,\y_k)=W_\muhat^\tomhat(t).
$$
\label{W}
\end{lemma}

\begin{proof}Since the coefficients of $V_{\alphahat_i}(t)$ are non-negative, there is an appropriate specialization of $V_{\alphahat_i}(1)$ variables in $\y_i=\{y_{i,1},y_{i,2},\dots\}$ into monomials $t^i$, with $i\geq 0$ (the other variables being specialized to $0$) such that 

$$
p_1(\y_i)=V_{\alphahat_i}(t).
$$
If $\omhat\in\overline{\mathfrak{H}}{^{-1}}(\tomhat)$ correspond to $(\lambda^1,\dots,\lambda^s)\in\calP_{n_1}\times\cdots\times\,\calP_{n_s}$ where $\lambda^i=(\lambda^i_1,\lambda^i_2,\dots)$ then 

\begin{align*}
V_\omhat(t)&=\prod_{i=1}^s\prod_jV_{\alphahat_i}\left(t^{\lambda^i_j}\right)\\
&=\prod_{i,j}p_1(\y_i^{\lambda^i_j})\\
&=\prod_{i=1}^sp_{\lambda^i}(\y_i).
\end{align*}
\end{proof}

\begin{remark} By the discussion above Lemma \ref{W}, note that 
$$
\sum_{\omhat\in \overline{\mathfrak{H}}{^{-1}}(\tomhat)}A_\omhat^oc_\omhat^\muhat=\left\langle \C_\tomhat^\muhat,1\right\rangle_{W_\tomhat}
$$
and so if  $V_{\alphahat_i}(t)=1$ for all $i=1,\dots,s$, then $W_\muhat^\tomhat(q)=\left\langle \C_\tomhat^\muhat,1\right\rangle_{W_\tomhat}$.

\label{inner}\end{remark}

We now decompose the character of the representation $W_\tomhat\rightarrow\GL(\C_\tomhat^\muhat)$ as a sum of irreducible characters

$$
\sum_{(\tau^1,\dots,\tau^s)\in \calP_{n_1}\times\cdots\times\,\calP_{n_s}}m_{(\tau^1,\dots,\tau^s)}\,\chi^{\tau^1}\cdots \chi^{\tau^s}.
$$We thus have 

\begin{align*}
W_\muhat^\tomhat(\y_1,\dots,\y_k)&=\sum_{(\tau^1,\dots,\tau^s)\in \calP_{n_1}\times\cdots\times\,\calP_{n_s}}m_{(\tau^1,\dots,\tau^s)}\,\sum_{(\lambda^1,\dots,\lambda^s)\in\calP_{n_1}\times\cdots\times\,\calP_{n_s}}\frac{1}{z_{\lambda^1}\cdots z_{\lambda^s}}\,p_{\lambda^1}(\y_1)\cdots p_{\lambda^s}(\y_s)\,\chi^{\tau^1}_{\lambda^1}\cdots\,\chi^{\tau^s}_{\lambda^s}.\\
&=\sum_{(\tau^1,\dots,\tau^s)\in \calP_{n_1}\times\cdots\times\,\calP_{n_s}}m_{(\tau^1,\dots,\tau^s)}\,\prod_{i=1}^s\sum_{\lambda\in\calP_{n_i}}\frac{1}{z_\lambda}\,p_\lambda(\y_i)\,\chi^{\tau^i}_\lambda\\
&=\sum_{(\tau^1,\dots,\tau^s)\in \calP_{n_1}\times\cdots\times\,\calP_{n_s}}m_{(\tau^1,\dots,\tau^s)}\,s_{\tau^1}(\y_1)\cdots s_{\tau^s}(\y_s).
\end{align*}
Now a Schur function $s_\lambda$ decomposes as 
$$
s_\lambda=\sum_{\mu\unlhd\lambda} K_{\lambda\mu}m_\mu
$$
where $m_\mu$ is the monomial symmetric function associated with $\mu$ and  $\{K_{\lambda\mu}\}_{\lambda,\mu}$ are the Kostka numbers which are non-negative integers. Hence 

\beq
W_\muhat^\tomhat(\y_1,\dots,\y_k)=\sum_{(\tau^1,\dots,\tau^s)\in \calP_{n_1}\times\cdots\times\,\calP_{n_s}}f_{(\tau^1,\dots,\tau^s)}\,m_{\tau^1}(\y_1)\cdots m_{\tau^s}(\y_s),
\label{form}
\eeq
for some $f_{(\tau^1,\dots,\tau^s)}\in\N$.

Put

$$
\calR_{\tomhat,\,\muhat}:=\{(\tau^1,\dots,\tau^s)\in\calP_{n_1}\times\cdots\times\,\calP_{n_s}\,|\,m_{(\tau^1,\dots,\tau^s)}\neq 0\}.
$$
For a partition $\lambda$, denote by $\ell(\lambda)$ its length. 

\begin{theorem}Assume that for all $i=1,2,\dots,s$, the polynomial $V_{\alphahat_i}(t)$ has non-negative integer coefficients. Then the polynomial $W_\muhat^\tomhat(t)$ has non-negative integer coefficients. Moreover it is non-zero if and only if there exists a sequence $(\tau^1,\dots,\tau^s)\in\calR_{\tomhat,\,\muhat}$ such that for all $i=1,\dots,s$, we have 

$$
\ell(\tau^i)\leq V_{\alphahat_i}(1).
$$
\label{THEOW}\end{theorem}

\begin{proof} The assertion (i) follows from the fact that if we specialize the variables $\y_i$ according to Lemma \ref{W} we see that the right hand side of Formula (\ref{form}) is a polynomial in $t$ with non-negative coefficients. We have 

$$
W_\muhat^\tomhat(t)=\sum_{(\tau^1,\dots,\tau^s)\in \calP_{n_1}\times\cdots\times\,\calP_{n_s}}m_{(\tau^1,\dots,\tau^s)}\,s_{\tau^1}(\y_1)\cdots s_{\tau^s}(\y_s)\neq 0
$$
if and only if there exists $(\tau_1,\dots,\tau_s)\in\calR_{\tomhat,\,\muhat}$ such that for all $i=1,\dots,s$ we have $s_{\tau^i}(\y_i)\neq 0$. But $s_{\tau^i}(\y_i)\neq 0$ if and only if there exists a partition $\lambda^i\unlhd \tau^i$ such that $m_{\lambda^i}(\y_i)\neq 0$, i.e., such that $\ell(\lambda^i)\leq V_{\alphahat_i}(1)$. Indeed, the integer $V_{\alpha_i}(1)$ is the number of variables in $\y_i$ that are specialized to a monomial $q^i$, the other variables being specialized to $0$. We conclude by noticing that if $\lambda^i\unlhd\tau^i$, then $\ell(\tau^i)\leq\ell(\lambda^i)$.

\end{proof}

For simplicity choose a total ordering $\geq$ on $\oP$ and denote the elements of $\otT$ in the form $\tomhat=\alphahat_1^{n_1}\alphahat_2^{n_2}\cdots\alphahat_s^{n_s}$ with $\alphahat_1\geq\alphahat_2\geq\cdots\geq \alphahat_s$  and $\tomhat(\alphahat_i)=n_i$.

\begin{theorem} Assume that the polynomials $V_\alphahat(t)$, with $|\alphahat|\leq n$, have non-negative integer coefficients. Then we have the following assertions.

\noindent (i) For any $\muhat\in\oP_n$, the polynomial $U_\muhat(t)$ has non-negative integer coefficients.

\noindent (ii) The polynomial $U_\muhat(t)$ is non-zero if and only if there exists $\tomhat=\alphahat_1^{n_1}\alphahat_2^{n_2}\cdots\alphahat_s^{n_s}\in\otT$  and a sequence $(\tau^1,\dots,\tau^s)\in\calR_{\tomhat,\,\muhat}$ such that for all $i=1,\dots,s$, we have 

$$
\ell(\tau^i)\leq V_{\alphahat_i}(1).
$$
\label{THEO1}\end{theorem}

\begin{proof} Follows from Theorem \ref{THEOW} and the fact that the polynomial  $U_\muhat(t)$ is non-zero if and only if  there exists an $\tomhat\in\otT$ such that $W_\muhat ^\tomhat(t)\neq 0$. \end{proof}

\subsection{Cauchy function}\label{cauchy}

Given a partition $\lambda\in\calP_n$ and an integer $g\geq 0$, we
define 
\begin{equation}
\calH_\lambda(t):=\frac{t^{g\,\langle\lambda,\lambda\rangle}}{a_\lambda(t)}.
\label{hookpure}\end{equation}
where $a_\lambda(q)$ is the cardinality of the centralizer in $\GL_{|\lambda|}(\F_q)$ of a unipotent element
with Jordan form of type $\lambda$. 

For a partition $\lambda$, let $\tH_\lambda(\x;t)=\sum_\nu\tilde{K}_{\nu\lambda}(t)s_\nu(\x)\in\Lambda(\x)\otimes_\Z\Q(t)$ where $\tilde{K}_{\nu\lambda}(t)$ are the modified Kostka polynomials \cite[Chapter III, \S 7]{macdonald}.

As in \cite{hausel-letellier-villegas} we consider the function $\Omega(t)\in\Lambda[[T]]$ defined as

$$
\Omega(t)=\Omega(\x_1,\dots,\x_k;t):=\sum_{\lambda\in \calP} \calH_{\lambda}(t)
\prod_{i=1}^k\tilde{H_\lambda}(\x_i;t).
$$
When $g=0$ and $k=2$ this function was first considered by Garsia and Haiman \cite{garsia-haiman}.  Define $\bbV(t)\in T\Lambda[[T]]$ and $\bbU(t)\in 1+T\Lambda[[T]]$ by

$$
\bbV(t):=(t-1)\,\Log\,\Omega(t),\hspace{.5cm} \bbU(t):=\Exp\,\bbV(t)
$$

\subsection{Harcos inequality}\label{harcossec}

We extend the definition of size of partitions to any sequence $\x=(x_1,\dots,x_k)$ of non-negative integers as 

$$
|\x|:=\sum_i x_i.
$$
For two sequences $\c=(c_1,\dots,c_r)$, $\x=(x_1,\dots,x_s)$, with $c:={\rm max}_i\, c_i$ define 

$$
\sigma_\c(\x):=c|\x|^2-|\c|\sum_ix_i^2.
$$

Let us now state Harcos theorem \cite[Appendix]{hausel-letellier-villegas2}.

\begin{theorem} Let $r,s>0$ be integers. For $i=1,\dots,s$, let $\x^i=(x_1^i,\dots,x_r^i)$ be a sequence with non-negative numbers. Put $\c=(c_1,\dots,c_r):=\sum_{i=1}^s\x^i$  and $c:={\rm max}_i\, c_i$. Then 
$$
\sigma_\c(\c)\geq \sum_{i=1}^s \sigma_\c(\x^i).
$$\label{harcos}
\end{theorem}

\begin{corollary}Let $\alpha^1,\dots,\alpha^s$ be $s$ partitions and let $\mu$ be a partition of size $|\sum_i\alpha^i|$  such that $\mu\unlhd\sum_i\alpha^i$. Then 

$$
\sigma_\mu(\mu)\geq \sum_{i=1}^s\sigma_\mu(\alpha^i).
$$
\end{corollary}

\begin{proof} We have 

$$
\sigma_\mu(\mu)-\sum_{i=1}^s\sigma_\mu(\alpha^i)=\sigma_\mu(\mu)-\sum_{i=1}^s\left(\mu_1|\alpha^i|^2-|\mu|\sum_j(\alpha^i_j)^2\right).$$
Since $\mu\unlhd\sum_i\alpha^i$, we can find $\x^1,\dots,\x^k$ such that 

1) $\mu=\sum_i\x^i$,

2) for all $i=1,\dots,s$,  $|\x^i|=|\alpha^i|$ and $\sum_j(x^i_j)^2\leq\sum_j(\alpha^i_j)^2$.

Notices that the sequences $\x^i$ may not be partitions any more.

Hence 

$$\sigma_\mu(\mu)-\sum_{i=1}^s\sigma_\mu(\alpha^i)\geq \sigma_\mu(\mu)-\sum_{i=1}^s\sigma_\mu(\x^i)$$
which is non-negative by Theorem \ref{harcos}.

\end{proof}

\section{Tensor products of unipotent characters}

On $\Lambda$ we put $\langle\,,\,\rangle:=\prod_i\langle\,,\,\rangle_i$ where $\langle\,,\,\rangle_i$ denotes the Hall pairing on $\Lambda(\x_i)$ which makes  the basis $\{s_\mu(\x_i)\}$ of Schur symmetric functions  orthonormal.

\subsection{Irreducible characters of unipotent type}

To alleviate the notation, put $G:=\GL_n(\F_q)$. Let $B\subset G$ be the upper triangular matrices and let $\C[G/B]$ be the $\C$-vector space with basis $G/B=\{gB\,|\,g\in G\}$. The group $G$ acts on $\C[G/B]$ by left multiplication. Let us denote by ${\rm Ind}_B^G(1):G\rightarrow\C, g\mapsto {\rm Trace}\,\left(g\,|\,\C[G/B]\right)$ the character of the representation $G\rightarrow\GL\left(\C[G/B]\right)$. The decomposition of ${\rm Ind}_B^G(1)$ as a sum of irreducible charaters of $G$ reads

$$
{\rm Ind}_B^G(1)=\sum_{\chi\in {\rm Irr}\,\calS_n}\chi(1)\cdot \calU_\chi.
$$
The irreducible characters $\{\calU_\chi\}_\chi$ are called the \emph{unipotent} characters of $G$. The character $\calU_1$ is the trivial character of $G$ and $\calU_\epsilon$, where $\epsilon$ is the sign character of $\calS_n$, is the Steinberg character of $G$. For a partition $\lambda$ of $n$, we put 

$$
\calU_\lambda:=\calU_{\chi^\lambda}
$$
so that the $\calU_{(1^n)}$ is the Steinberg character and $\calU_{(n^1)}$ is the trivial character.

We say that an irreducible character of $G$ is of \emph{unipotent type} if it is of the form 
$(\alpha\circ{\rm det})¥\cdot\calU_\lambda$ for some partition $\lambda$ and some linear character $\alpha:\F_q^\times\rightarrow\C^\times$.

\subsection{Comet-shaped quivers}\label{quiver}

Given a non-negative integer $g$ and a
$k$-tuple $\muhat=(\mu^1,\mu^2,\dots,\mu^k)\in\oP_n$ with $n\geq 1$, we denote by $\Gamma_\muhat$ 
 the \emph{comet-shaped} quiver

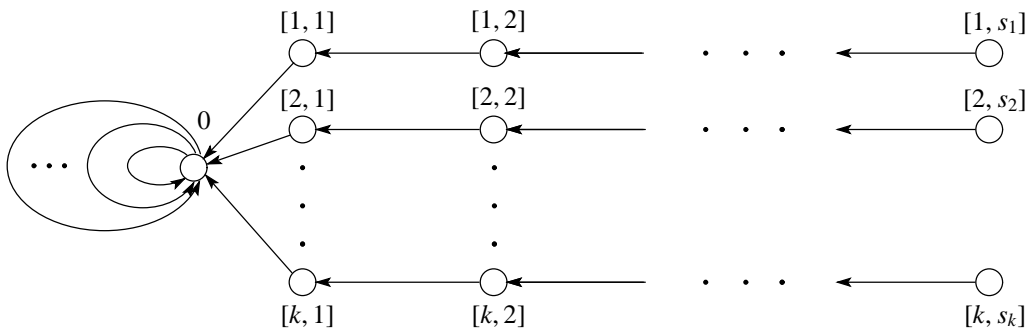
\begin{center}
\unitlength 0.1in
\begin{picture}( 52.1000, 15.4500)(  4.0000,-17.0000)
%
\special{pn 8}%
\special{ar 1376 1010 70 70  0.0000000 6.2831853}%
%
\special{pn 8}%
\special{ar 1946 410 70 70  0.0000000 6.2831853}%
%
\special{pn 8}%
\special{ar 2946 410 70 70  0.0000000 6.2831853}%
%
\special{pn 8}%
\special{ar 5540 410 70 70  0.0000000 6.2831853}%
%
\special{pn 8}%
\special{ar 1946 810 70 70  0.0000000 6.2831853}%
%
\special{pn 8}%
\special{ar 2946 810 70 70  0.0000000 6.2831853}%
%
\special{pn 8}%
\special{ar 5540 810 70 70  0.0000000 6.2831853}%
%
\special{pn 8}%
\special{ar 1946 1610 70 70  0.0000000 6.2831853}%
%
\special{pn 8}%
\special{ar 2946 1610 70 70  0.0000000 6.2831853}%
%
\special{pn 8}%
\special{ar 5540 1610 70 70  0.0000000 6.2831853}%
%
\special{pn 8}%
\special{pa 1890 1560}%
\special{pa 1440 1050}%
\special{fp}%
\special{sh 1}%
\special{pa 1440 1050}%
\special{pa 1470 1114}%
\special{pa 1476 1090}%
\special{pa 1500 1088}%
\special{pa 1440 1050}%
\special{fp}%
%
\special{pn 8}%
\special{pa 2870 410}%
\special{pa 2020 410}%
\special{fp}%
\special{sh 1}%
\special{pa 2020 410}%
\special{pa 2088 430}%
\special{pa 2074 410}%
\special{pa 2088 390}%
\special{pa 2020 410}%
\special{fp}%
%
\special{pn 8}%
\special{pa 3720 410}%
\special{pa 3010 410}%
\special{fp}%
\special{sh 1}%
\special{pa 3010 410}%
\special{pa 3078 430}%
\special{pa 3064 410}%
\special{pa 3078 390}%
\special{pa 3010 410}%
\special{fp}%
\special{pa 3730 410}%
\special{pa 3010 410}%
\special{fp}%
\special{sh 1}%
\special{pa 3010 410}%
\special{pa 3078 430}%
\special{pa 3064 410}%
\special{pa 3078 390}%
\special{pa 3010 410}%
\special{fp}%
%
\special{pn 8}%
\special{pa 2870 810}%
\special{pa 2020 810}%
\special{fp}%
\special{sh 1}%
\special{pa 2020 810}%
\special{pa 2088 830}%
\special{pa 2074 810}%
\special{pa 2088 790}%
\special{pa 2020 810}%
\special{fp}%
%
\special{pn 8}%
\special{pa 2870 1610}%
\special{pa 2020 1610}%
\special{fp}%
\special{sh 1}%
\special{pa 2020 1610}%
\special{pa 2088 1630}%
\special{pa 2074 1610}%
\special{pa 2088 1590}%
\special{pa 2020 1610}%
\special{fp}%
%
\special{pn 8}%
\special{pa 3730 810}%
\special{pa 3020 810}%
\special{fp}%
\special{sh 1}%
\special{pa 3020 810}%
\special{pa 3088 830}%
\special{pa 3074 810}%
\special{pa 3088 790}%
\special{pa 3020 810}%
\special{fp}%
\special{pa 3740 810}%
\special{pa 3020 810}%
\special{fp}%
\special{sh 1}%
\special{pa 3020 810}%
\special{pa 3088 830}%
\special{pa 3074 810}%
\special{pa 3088 790}%
\special{pa 3020 810}%
\special{fp}%
%
\special{pn 8}%
\special{pa 3730 1610}%
\special{pa 3020 1610}%
\special{fp}%
\special{sh 1}%
\special{pa 3020 1610}%
\special{pa 3088 1630}%
\special{pa 3074 1610}%
\special{pa 3088 1590}%
\special{pa 3020 1610}%
\special{fp}%
\special{pa 3740 1610}%
\special{pa 3020 1610}%
\special{fp}%
\special{sh 1}%
\special{pa 3020 1610}%
\special{pa 3088 1630}%
\special{pa 3074 1610}%
\special{pa 3088 1590}%
\special{pa 3020 1610}%
\special{fp}%
%
\special{pn 8}%
\special{pa 5466 410}%
\special{pa 4746 410}%
\special{fp}%
\special{sh 1}%
\special{pa 4746 410}%
\special{pa 4812 430}%
\special{pa 4798 410}%
\special{pa 4812 390}%
\special{pa 4746 410}%
\special{fp}%
%
\special{pn 8}%
\special{pa 5466 810}%
\special{pa 4746 810}%
\special{fp}%
\special{sh 1}%
\special{pa 4746 810}%
\special{pa 4812 830}%
\special{pa 4798 810}%
\special{pa 4812 790}%
\special{pa 4746 810}%
\special{fp}%
%
\special{pn 8}%
\special{pa 5466 1610}%
\special{pa 4746 1610}%
\special{fp}%
\special{sh 1}%
\special{pa 4746 1610}%
\special{pa 4812 1630}%
\special{pa 4798 1610}%
\special{pa 4812 1590}%
\special{pa 4746 1610}%
\special{fp}%
%
\special{pn 8}%
\special{pa 1880 840}%
\special{pa 1450 990}%
\special{fp}%
\special{sh 1}%
\special{pa 1450 990}%
\special{pa 1520 988}%
\special{pa 1500 972}%
\special{pa 1506 950}%
\special{pa 1450 990}%
\special{fp}%
%
\special{pn 8}%
\special{pa 1900 460}%
\special{pa 1430 960}%
\special{fp}%
\special{sh 1}%
\special{pa 1430 960}%
\special{pa 1490 926}%
\special{pa 1468 922}%
\special{pa 1462 898}%
\special{pa 1430 960}%
\special{fp}%
%
\special{pn 8}%
\special{sh 1}%
\special{ar 1946 1010 10 10 0  6.28318530717959E+0000}%
\special{sh 1}%
\special{ar 1946 1210 10 10 0  6.28318530717959E+0000}%
\special{sh 1}%
\special{ar 1946 1410 10 10 0  6.28318530717959E+0000}%
\special{sh 1}%
\special{ar 1946 1410 10 10 0  6.28318530717959E+0000}%
%
\special{pn 8}%
\special{sh 1}%
\special{ar 4056 410 10 10 0  6.28318530717959E+0000}%
\special{sh 1}%
\special{ar 4266 410 10 10 0  6.28318530717959E+0000}%
\special{sh 1}%
\special{ar 4456 410 10 10 0  6.28318530717959E+0000}%
\special{sh 1}%
\special{ar 4456 410 10 10 0  6.28318530717959E+0000}%
%
\special{pn 8}%
\special{sh 1}%
\special{ar 4056 810 10 10 0  6.28318530717959E+0000}%
\special{sh 1}%
\special{ar 4266 810 10 10 0  6.28318530717959E+0000}%
\special{sh 1}%
\special{ar 4456 810 10 10 0  6.28318530717959E+0000}%
\special{sh 1}%
\special{ar 4456 810 10 10 0  6.28318530717959E+0000}%
%
\special{pn 8}%
\special{sh 1}%
\special{ar 4056 1610 10 10 0  6.28318530717959E+0000}%
\special{sh 1}%
\special{ar 4266 1610 10 10 0  6.28318530717959E+0000}%
\special{sh 1}%
\special{ar 4456 1610 10 10 0  6.28318530717959E+0000}%
\special{sh 1}%
\special{ar 4456 1610 10 10 0  6.28318530717959E+0000}%
\put(19.7000,-2.4500){\makebox(0,0){$[1,1]$}}%
\put(29.7000,-2.4000){\makebox(0,0){$[1,2]$}}%
\put(55.7000,-2.5000){\makebox(0,0){$[1,s_1]$}}%
\put(19.7000,-6.5500){\makebox(0,0){$[2,1]$}}%
\put(29.7000,-6.4500){\makebox(0,0){$[2,2]$}}%
\put(55.7000,-6.5500){\makebox(0,0){$[2,s_2]$}}%
\put(19.7000,-17.8500){\makebox(0,0){$[k,1]$}}%
\put(29.7000,-17.8500){\makebox(0,0){$[k,2]$}}%
\put(55.7000,-17.8500){\makebox(0,0){$[k,s_k]$}}%
\put(14.3000,-7.6000){\makebox(0,0){$0$}}%
\special{pn 8}%
\special{sh 1}%
\special{ar 2950 1010 10 10 0  6.28318530717959E+0000}%
\special{sh 1}%
\special{ar 2950 1210 10 10 0  6.28318530717959E+0000}%
\special{sh 1}%
\special{ar 2950 1410 10 10 0  6.28318530717959E+0000}%
\special{sh 1}%
\special{ar 2950 1410 10 10 0  6.28318530717959E+0000}%
\special{pn 8}%
\special{ar 1110 1000 290 220  0.4187469 5.9693013}%
\special{pn 8}%
\special{pa 1368 1102}%
\special{pa 1376 1090}%
\special{fp}%
\special{sh 1}%
\special{pa 1376 1090}%
\special{pa 1324 1138}%
\special{pa 1348 1136}%
\special{pa 1360 1158}%
\special{pa 1376 1090}%
\special{fp}%
\special{pn 8}%
\special{ar 910 1000 510 340  0.2464396 6.0978374}%
\special{pn 8}%
\special{pa 1400 1096}%
\special{pa 1406 1084}%
\special{fp}%
\special{sh 1}%
\special{pa 1406 1084}%
\special{pa 1362 1138}%
\special{pa 1384 1132}%
\special{pa 1398 1152}%
\special{pa 1406 1084}%
\special{fp}%
\special{pn 8}%
\special{sh 1}%
\special{ar 540 1000 10 10 0  6.28318530717959E+0000}%
\special{sh 1}%
\special{ar 620 1000 10 10 0  6.28318530717959E+0000}%
\special{sh 1}%
\special{ar 700 1000 10 10 0  6.28318530717959E+0000}%
\special{pn 8}%
\special{ar 1200 1000 170 100  0.7298997 5.6860086}%
\special{pn 8}%
\special{pa 1314 1076}%
\special{pa 1328 1068}%
\special{fp}%
\special{sh 1}%
\special{pa 1328 1068}%
\special{pa 1260 1084}%
\special{pa 1282 1094}%
\special{pa 1280 1118}%
\special{pa 1328 1068}%
\special{fp}%
\end{picture}%
\end{center}

\vspace{10pt}
\noindent with $k$ legs of
length $s_1,s_2,\dots,s_k$ (where $s_i=\ell(\mu^i)-1$) and with $g$ loops
at the central vertex. The
multi-partition $\muhat$ defines also a dimension vector $\v_\muhat$
of $\Gamma_\muhat$ whose coordinates on the $i$-th leg are
$(n,n-\mu^i_1,n-\mu^i_1-\mu^i_2,\dots,n-\sum_{r=1}^{s_i}\mu^i_r)$.

Let $I_\muhat=\{0\}\cup\{[i,j]\,|\,1\geq i\geq k,\, 1\geq j\geq s_i\}$ be the set of vertices of $\Gamma_\muhat$ and let ${\bf C}_\muhat=(c_{ij})_{i,j}$ be the Cartan matrix of $\Gamma_\muhat$, namely
$$
c_{ij}=\begin{cases} 2-2(\text{the number of edges joining }i\text{ to itself})\hspace{.2cm}\text{if }i=j\\ 
  - (\text{the number of edges joining }i\text{ to }j)\hspace{1.2cm}\text{ otherwise}.
         \end{cases}
$$
Let $(\,,\,)$ be the symmetric bilinear form on $\Z^{I_\muhat}$  defined by 

$$
(\e_i,\e_j)=c_{ij}.
$$
where for $i\in I_\muhat$, we denote by $\e_i$ the root of $\Gamma_\muhat$ with all zero coordinates except for a $1$ at the indicated vertex $i$. If there is no-edge loop at the vertex $i$, we say that $\e_i$ is a \emph{fundamental root} \cite[Chapter 1]{Kac1}. For a fundamental root $\e_i$ we define the associated fundamental reflection $s_i:\Z^{I_\muhat}\rightarrow\Z^{I_\muhat}$ by 

$$
s_i(\lambda)=\lambda-2(\lambda,\e_i)\,\e_i
$$
for all $\lambda\in\Z^I$. The group $W(\Gamma_\muhat)$ generated by all fundamental reflections is called the Weyl group of $\Gamma_\muhat$. A vector $\v\in\Z^{I_\muhat}$ is called a \emph{real root} of $\Gamma_\muhat$ if it is of the form $w(\e_i)$ for some fundamental root $\e_i$ and some $w\in W(\Gamma_\muhat)$. 
Recall \cite[Chapter 1]{Kac1} that  the fundamental set $M(\Gamma_\muhat)$ of imaginary roots is the set of vectors $\v\in(\N)^{I_\muhat}-\{0\}$ with connected support such that for all fundamental root $\e$, we have

$$
(\e,\v)\leq 0.
$$
The \emph{imaginary roots} are the vectors $\v\in\Z^{I_\muhat}$ which are of the form $w(\delta)$ or $w(-\delta)$ for some $\delta\in M(\Gamma_\muhat)$ and $w\in W(\Gamma_\muhat)$.

Let $\Phi(\Gamma_\muhat)\subset \Z^I$ be the set of all roots (real and imaginary) of $\Gamma_\muhat$ and
let $\Phi(\Gamma_\muhat)^+\subset \left(\N\right)^{I_\muhat}$ be the subset of positive
roots.

For $\alphahat=(\alpha^1,\dots,\alpha^k)\in\oP$, define

\begin{equation}
\delta(\alphahat):=(2g-2+k)n-\sum_{i=1}^k\alpha^i_1.
\end{equation}

\begin{proposition} A dimension vector $\v$ of $\Gamma_\muhat$ is in $M(\Gamma_\muhat)$ if and only if there exists $\alphahat\in\oP$ such that $\v=\v_\alphahat$ and $\delta(\alphahat)\geq 0$. \label{deltapos}\end{proposition}

\begin{proof} For all $i,j\geq 1$, we have 

$$
(\v,\e_{[i,j]})=-\left((v_{[i,j-1]}-v_{[i,j]})-(v_{[i,j]}-v_{[i,j+1]})\right)
$$
where for convenience $[i,0]$ denotes also the central vertex $0$. We also have 

\begin{align*}
(\v,\e_0)&=(2-2g)v_0-\sum_{i=1}^kv_{[i,1]}\\
&=-\left((2g-2+k)v_0-\sum_{i=1}^k(v_0-v_{[i,1]})\right)
\end{align*}
For all $i=1,\dots,k$, put $\alpha^i_1:=v_0-v_{[i,1]}\in\Z$ and $\alpha^i_j:=v_{[i,j-1]}-v_{[i,j]}$ for all $j\geq 1$, and put $\alpha^i=(\alpha^i_1,\alpha^i_2,\dots)$.

Then $\v$ is in the fundamental domain if and only if for all $i=1,\dots,k$, the tuple $\alpha^i$ is a partition and $\delta(\alphahat)\geq 0$ where $\alphahat=(\alpha^1,\dots,\alpha^k)$. It is also clear that $\v=\v_\alphahat$.

\end{proof}

Note that if $g\geq 1$, then $\delta(\muhat)\geq 0$ and so $\v_\muhat$ is always an imaginary root.

Put 

\begin{equation}
A_\muhat(t):=\left\langle \bbV(t),h_\muhat\right\rangle
\label{h}\end{equation}
where $\bbV(t)$ is as in \S \ref{cauchy} and $h_\muhat=h_{\mu^1}(\x_1)\cdots h_{\mu^k}(\x_k)$ denotes the complete symmetric function.

For $\muhat\in\oP$, put

\begin{equation}
d_\muhat:=n^2(2g-2+k)-\sum_{i,j}(\mu^i_j)^2+2=2-{^t}\v_\muhat{\bf C}_\muhat\v_\muhat.
\label{dmu}\end{equation}

Recall one of the main result of  \cite{hausel-letellier-villegas2}.

\begin{theorem} (i) For any finite field $\F_q$, the evaluation $A_\muhat(q)$ counts the number of isomorphism classes of absolutely indecomposable representations of $\Gamma_\muhat$ of dimension $\v_\muhat$ over $\F_q$.

\noindent (ii) If non-zero, $A_\muhat(t)$ is a monic polynomial of degree $d_\muhat/2$ with integer coefficients.

\noindent (iii) The polynomial $A_\muhat(t)$ is non-zero if and only if $\v_\muhat\in\Phi(\Gamma_\muhat)$. Moreover $A_\muhat(t)=1$ if and only if $\v_\muhat$ is a real root.

\label{kactheo}\end{theorem}

We recently proved in \cite{hausel-letellier-villegas3}  that the coefficients of $A_\muhat(t)$ are actually non-negative. The assertion (ii) follows from (i) using the results in \cite[\S 1.15]{kac}, and the assertion (iii) follows from (i) and \cite[\S 1.10]{kac}.

\subsection{The generic case}\label{generic}

Let $\muhat=(\mu^1,\dots,\mu^k)\in\oP_n$ with $n\geq 1$. A tuple $(\calX_1,\dots,\calX_k)$ of irreducible characters of $G$ is said to be of type $\muhat$ if for each $i=1,2,\dots,k$, there exists a linear character $\alpha_i:\F_q^\times\rightarrow\C^\times$ such that

$$
\calX_i:=(\alpha_i\circ{\rm det})\cdot\calU_{\mu^i}.
$$

The tuple $(\calX_1,\dots,\calX_k)$ said to be \emph{generic} if the size of the subgroup of ${\rm Irr}\,(\F_q^\times)$ generated by $\alpha_1\cdots\alpha_k$ equals $n$ (see  \cite[Definition 6.8.6]{letellier4}).

Fix an integer $g\geq 0$ and consider $\calE:G\rightarrow\C$, $x\mapsto q^{g\,{\rm dim}\,C_{\GL_n}(x)}$. If $g=1$, note that $\calE$ is the character of the representation of $G$ in the group algebra $\C[\g]$ where $G$ acts on $\g:=\gl_n(\F_q)$ by conjugation.

Define

\begin{equation}
V_\muhat(t):=\left\langle\bbV(t),s_\muhat\right\rangle.
\label{v}\end{equation}
I.e., the $V_\muhat(t)$ are defined by the identity

$$
\sum_{\muhat\in\oP}V_\muhat(t)s_\muhat=\bbV(t).
$$

Recall the following theorem \cite[\S 6.10.6]{letellier4}.

\begin{theorem} For any generic tuple $(\calX_1,\dots,\calX_k)$ of type $\muhat$ we have 

\begin{equation}
\left\langle\calE\otimes\calX_1\otimes\cdots\otimes\calX_k,1\right\rangle_G=V_\muhat(q)\label{multi}
\end{equation}
\label{multit}
\end{theorem}

By definition $V_\muhat(t)$ is a rational function in $t$ with rational coefficients and by the above theorem it is an integer for infinitely many values of $t$. Hence $V_\muhat(t)$ is a polynomial in $t$  with rational coefficients.

\begin{theorem} (1) The polynomial $V_\muhat(t)$ is non-zero if and only if $\v_\muhat\in\Phi(\Gamma_\muhat)$. Moreover $V_\muhat(t)=1$ if and only if $\v_\muhat$ is a real root.

\noindent (2) If non-zero, the polynomial $V_\muhat(t)$ is a monic polynomial of degree $d_\muhat/2$ with non-negative integer coefficients.
\label{theogeneric}
\end{theorem}

For $\alphahat=(\alpha^1,\dots,\alpha^k),\betahat=(\beta^1,\dots,\beta^k)\in\oP_n$, say that $\alphahat\unlhd\betahat$ if  $\alpha^i\unlhd\beta^i$ for all $i=1,\dots,k$. We will need the following lemma.

\begin{lemma} If $\alphahat,\betahat\in\oP_n$ are such that $\alphahat\unlhd\betahat$ and $\alphahat\neq\betahat$ then $d_\betahat< d_\alphahat$.
\label{lem1}
\end{lemma}\begin{proof} We need to see that for two partitions $\lambda=(\lambda_1,\dots,\lambda_r)$ and $\mu=(\mu_1,\dots,\mu_s)$ such that $\lambda\unlhd\mu$ and $\lambda\neq\mu$, we have $\sum_{i=1}^r\lambda_i^2<\sum_{i=1}^s\mu_i^2$. This follows from the formula$$\sum_{i=1}^l\mu_i^2-\sum_{i=1}^l\lambda_i^2=(\mu_1-\lambda_1)(\mu_1-\mu_2+\lambda_1-\lambda_2)+(\mu_1+\mu_2-\lambda_1-\lambda_2)(\mu_2-\mu_3+\lambda_2-\lambda_3)+\cdots+\left(\sum_{i=1}^l\mu_i-\sum_{i=1}^l\lambda_i\right)(\mu_s+\lambda_s)$$which is available for all $l$ (with the convention that $\lambda_i=0$ if $i>r$).\end{proof}

\begin{proof}[Proof of Theorem \ref{theogeneric}] We have the following relations between Schur and complete symmetric functions \cite[page 101]{macdonald}:\begin{equation*}h_\lambda=\sum_{\mu\unrhd\lambda} K_{\lambda\mu}'s_\mu,\hspace{.5cm}s_\mu=\sum_{\lambda\unrhd\mu} K^*_{\mu\lambda}h_\lambda,\end{equation*}where $K=(K_{\lambda\mu})_{\lambda,\mu}$ is the matrix whose coefficients are Kostka numbers, $K'=(K_{\lambda\mu}')_{\lambda,\mu}$ is the transpose of $K$ and $K^*=(K^*_{\lambda\mu})_{\lambda,\mu}$ is the transpose inverse of $K$. By Formulas (\ref{h}) and (\ref{v}) we have for any $\lambdahat,\muhat\in\oP-\{0\}$\begin{equation}A_\lambdahat(t)=\sum_{\muhat\unrhd\lambdahat} K_{\lambdahat\muhat}'V_\muhat(t),\hspace{.5cm}V_\muhat(t)=\sum_{\lambdahat\unrhd\muhat} K^*_{\muhat\lambdahat}A_\lambdahat(t),\label{eq1}\end{equation}where $K_{\lambdahat\muhat}:=\prod_{i=1}^kK_{\lambda^i\mu^i}$. Assume that $\v_\muhat\in\Phi(\Gamma_\muhat)$. By Theorem \ref{kactheo} the polynomial $A_\muhat(t)$ is monic of degree $d_\muhat/2$. Moreover $K^*_{\muhat\muhat}=1$ and by Lemma \ref{lem1} the degree of the polynomials $A_\lambdahat(t)$, with $\lambdahat\unrhd\muhat,\lambdahat\neq\muhat$, are of degree strictly smaller than $d_\muhat/2$. Hence we deduce from the second formula  (\ref{eq1}) that $V_\muhat(t)$ is non-zero and is a monic polynomial of degree $d_\muhat/2$. Note that if $\v_\muhat$ is real, then $A_\muhat(t)=1$ and $A_\lambdahat(t)=0$ if $\lambdahat\unrhd\muhat,\lambdahat\neq\muhat$ as $d_\lambdahat<d_\muhat=0$, and so $V_\muhat(t)=1$. Assume now that $V_\muhat(t)\neq 0$. Recall that $K_{\lambdahat\muhat}$ are non-negative integers and that $K_{\muhat\muhat}=1$. Moreover, for all $\alphahat\in\oP-\{0\}$,  the evaluation $V_\alphahat(q)$ of $V_\alphahat(t)$ at $q$ is a non-negative integer by Theorem \ref{multit}. Hence, by the first formula (\ref{eq1}), the polynomial $A_\muhat(t)$ must be non-zero and so, by Theorem \ref{kactheo}, the dimension vector $\v_\muhat$ is a root.

Let us now outline the proof of the positivity which is similar to the proof of the main result of \cite{hausel-letellier-villegas3}. Denote by $\K$ an arbitrary algebraic closure of $\F_q$ and put $\gl_n:=\gl_n(\K)$. Denote by $F:\gl_n\rightarrow\gl_n$ the Frobenius endomorphism that raises coefficients of matrices to their $q$-th power. Say that a tuple $(\calC_1,\dots,\calC_k)$ of adjoint orbits of $\gl_n$ is \emph{generic} \cite[\S 5.1]{letellier4} if $\sum_{i=1}^k{\rm Tr}\,(\calC_i)=0$ and if for any subspace $V\subset\K^n$ stable by some $X_i\in\calC_i$ for each $i=1,\dots,k$, such that

$$
\sum_{i=1}^k{\rm Tr}\,(X_i|_V)=0
$$
then either $V=0$ or $V=\K^n$. Generic tuples of semisimple regular adjoint orbits always exists \cite[\S 2.2]{hausel-letellier-villegas}\cite[\S 5.1]{letellier4}. Recall that the $G$-conjugacy classes of $F$-stable maximal tori of $\GL_n=\GL_n(\K)$ are parametrized by the conjugacy classes of $\mathfrak{S}_n$. For $w$ in the symmetric group  $\mathfrak{S}_n$ we denote by $T_w$ a representative of the corresponding $G$-conjugacy class of maximal tori. Say that an $F$-stable regular semisimple adjoint orbit of $\gl_n$ is of type $w\in\mathfrak{S}_n$ if it has a non-empty intersection with $\t_w^F$ where $\t_w:={\rm Lie}(T_w)$. Denote by $\mathbb{S}_n=\mathfrak{S}_n\times\cdots\times\mathfrak{S}_n$ the Cartesian product of $k$ copies of $\mathfrak{S}_n$. Now for each conjugacy class of $\mathbb{S}_n$ with representative $\w=(w_1,\dots,w_k)$ choose a generic tuple $(\calC^{w_1},\dots,\calC^{w_k})$ of $F$-stable semisimple regular adjoint orbits of $\gl_n$ of type $\w$ (such a choice is possible for any $\w$ assuming that $q$ is sufficiently large which we now assume). Consider the space

$$
\calV^\w:=\left\{(A_1,B_1,\dots,A_g,B_g,X_1,\dots,X_k)\in\gl_n^{2g}\times\calC^{w_1}\times\cdots\times\calC^{w_k}\,\left|\,\sum_i[A_i,B_i]+\sum_j X_j=0\right.\right\},
$$
and the affine GIT quotient

$$
\calQ^\w:=\calV^\w/\!/\GL_n:={\rm Spec}\,\left(\K[\calV^\w]^{\GL_n}\right)
$$
where $\GL_n$ acts on $\calV^\w$ diagonally by conjugation. It is well-known (see for instance \cite{hausel-letellier-villegas}) that $\calQ^\w$ is non-singular, irreducible and has vanishing odd cohomology. If $\w=1$, we will simply write $\calQ$ instead of $\calQ^1$. We know by Lemma 7.2.1, Theorem 6.9.1 and Theorem 6.10.1 in \cite{letellier4}  that 

\beq
\epsilon(\w_\lambdahat)\,\,\#\calQ^{\w_\lambdahat}(\F_q)=q^{d/2}\left\langle \bbV(q),p_\lambdahat\right\rangle
\label{poly}\eeq
where 

$$d:=n^2(2g-2+k)-kn+2$$
is the dimension of $\calQ$, $\epsilon$ is the sign character of $\mathbb{S}_n$ and where $\w_\lambdahat\in\mathbb{S}_n$ is in the conjugacy class corresponding to the multi-partition $\lambdahat\in\oP$. By definition $X_\lambda(t):=\left\langle \bbV(t),p_\lambdahat\right\rangle$ is a rational function in $t$ and by (\ref{poly}) the function $t^{d/2}X_\lambda(t)$ is an integer for infinitely many values of $t$, therefore $t^{d/2}X_\lambda(t)$ must by a polynomial in $t$ with rational coefficients.

On the other hand, exactly as in \cite[\S 2.2.1]{hausel-letellier-villegas3} we can prove (assuming that the characteristic is large enough) that there exists, for each $i$, a representation $\rho^i:\mathbb{S}_n\rightarrow\GL(H_c^{2i}(\calQ))$, where $H_c^{2i}(\calQ)$ denotes the compactly supported $\ell$-adic cohomology of $\calQ$, such that 

\beq
\#\calQ^\w(\F_q)=\sum_{i=d/2}^d{\rm Trace}\,\left(\rho^i(\w)\,|\, H_c^{2i}(\calQ)\right) q^i.
\label{W}\eeq
Note that the character $\psi^i:\mathbb{S}_n\rightarrow \Q$, $\w\mapsto {\rm Trace}\,\left(\rho^i(\w)\,|\, H_c^{2i}(\calQ)\right)$ of the representation $\rho^i$ does not depend on $q$. \emph{A priori} it should depend on the characteristic of $\K$ but it does not because it follows from the identities (\ref{poly}) and (\ref{W}) that the values of the characters of the representations $\rho^i$ are given by the coefficients of the polynomials $t^{d/2}X_\lambda(t)\in\Q[t]$. If $\calQ/\C$ denotes the complex analogue of $\calQ$ (see \ref{pos}), then we know by \cite[see below Lemma 48]{maffei} that there is a representation $\rho^i_\C$ of $W_\v$ on $H_c^{2i}(\calQ/\C,\C)$. We can actually prove (as in \cite[\S 2.2.1]{hausel-letellier-villegas3}) that  $\psi^i$ is also the character of $\rho^i_\C$. From the identity $s_\mu=\sum_\lambda z_\lambda^{-1}\chi^\mu_\lambda p_\lambda$ we find that 

$$
V_\muhat(q)=\sum_\lambdahat z_\lambdahat^{-1}\chi^\muhat_\lambdahat\left\langle\bbV(q),p_\lambdahat\right\rangle.$$

Combining Formulas (\ref{W}) and (\ref{poly}) we get that
\begin{align*}
V_\muhat(q)&=q^{-d/2}\sum_i \left(\sum_\lambdahat z_\lambdahat^{-1}\chi^\muhat_\lambdahat \epsilon(\w_\lambdahat) \psi^i(\w_\lambdahat)\right)q^i\\
&=q^{-d/2}\sum_i\left\langle \chi^{\muhat'},\psi^i\right\rangle_{\mathbb{S}_n} q^i,
\end{align*}
where $\muhat'$ denote the dual multi-partition of $\muhat$. Since these identities are true for infinitely many values of $q$, the coefficients of the polynomial  $V_\mu(t)$ coincide with the multiplicities $\left\langle \chi^{\muhat'},\psi^i\right\rangle_{\mathbb{S}_n}$ and therefore are non-negative integers.
\end{proof}

\subsection{The unipotent case}

For $\muhat\in\oP-\{0\}$, define

$$
U_\muhat(t):=\left\langle\bbU(t),s_\muhat\right\rangle
$$
where $\bbU(t)$ is as in \S \ref{cauchy} and put $U_\muhat(t):=1$ if $\muhat=0$.

\begin{proposition}For all $\muhat\in\oP-\{0\}$, we have

\begin{equation}
U_\muhat(q)=\left\langle\calE\otimes\calU_{\mu^1}\otimes\cdots\otimes\calU_{\mu^k},1\right\rangle.
\label{genvuni2}\end{equation}
\label{expgen}\end{proposition}

Recall the following relation between $V_\muhat(t)$ and $U_\muhat(t)$,
\begin{equation}
\sum_\muhat U_\muhat(t)s_\muhat=\Exp\left(\sum_\muhat V_\muhat(t) s_\muhat\right).
\label{genvuni}\end{equation}
To prove Proposition \ref{expgen} we  recall the definition of the type of a conjugacy class  of $G$. Let $F:\overline{\F}_q\rightarrow\overline{\F}_q$, $x\mapsto x^q$ be the Frobenius endomorphism and let $O$ be the set of $\langle F\rangle$-orbits of $\overline{\F}{^\times_q}$. The conjugacy classes of $G$ corresponds to the maps $f:O\rightarrow \P$ such that 

$$
\sum_{\gamma\in O}|\gamma|\cdot|f(\gamma)|=n.
$$
Let $C$ be a conjugacy class of $G$ corresponding to such a function $f$. The type of $C$ is the function $\omega_C=\omega_f\in\bT_n$ defined by 

$$
\omega_C(d,\lambda):=\#\left.\left\{\gamma\in O\,\right|\,(d,\lambda)=(|\gamma|,f(\gamma))\right\}.
$$

\begin{proof}[Proof of Proposition \ref{expgen}] We have

\begin{align*}
\left\langle\calE\otimes\calU_{\mu^1}\otimes\cdots\otimes\calU_{\mu^k},1\right\rangle&=\frac{1}{|G|}\sum_{g\in G}\calE(g)\calU_{\mu^1}(g)\cdots\calU_{\mu^k}(g)\\
&=\sum_C\frac{\calE (C )}{a_C(q)}\prod_{i=1}^k\calU_{\mu^i}(C )
\end{align*}
where the last sum is over the conjugacy classes of $G$ and where $a_C(q)$ denotes the cardinality of the centralizer in $G$ of an element of $C$. It is well-known (see for instance \cite[Theorem 2.2.2]{hausel-letellier-villegas2}) that for any conjugacy class $C$ of $G$ and any partition $\mu$ of $n$  we have

$$
\calU_\mu (C )=\left\langle \tH_{\omega_C}(\x;q),s_\mu(\x)\right\rangle.
$$
By Formula (\ref{hookpure}) we also have 

$$
\calH_{\omega_C}(q)=\frac{\calE (C )}{a_C(q)}.
$$

Let ${\bf C}_n$ be the set of conjugacy classes of $\GL_n(\F_q)$ and put ${\bf C}=\cup_{n\geq 1}{\bf C}_n$. Denote also by $\calP^O$ the set of all function $O\rightarrow\calP$ with finite support. If $0$ denotes the function that take the value $0$ everywhere, we put $\calH_{\omega_0}(q)=\tH_{\omega_0}(\x;q)=1$.

\begin{align*}
\sum_{\muhat\in\oP}\left\langle\calE\otimes\calU_{\mu^1}\otimes\cdots\otimes\calU_{\mu^k},1\right\rangle s_\muhat&=1+\sum_{C\in{\bf C}}\calH_{\omega_C}(q)\prod_i\tH_{\omega_C}(\x_i;q)\\
&=\sum_{f\in\calP^O}\calH_{\omega_f}(q)\prod_i\tH_{\omega_f}(\x_i;q)\\
&=\prod_{\gamma\in O}\Omega\left(\x_1^{|\gamma|},\dots\x_k^{|\gamma|};q^{|\gamma|}\right)\\
&=\prod_{d=1}^\infty\Omega\left(\x_1^d,\dots,\x_k^d;q^d\right)^{\phi_d(q)}
\end{align*}where $\phi_d(q)$ denotes the number of elements in $O$ of size $d$. Recall that 

$$ 
\phi_n(q)=\frac{1}{n}\sum_{d|n}\mu(d)(q^{n/d}-1).
$$
By Proposition \ref{moz} we deduce that 

$$
\Log\left(\sum_{\muhat\in\oP}\left\langle\calE\otimes\calU_{\mu^1}\otimes\cdots\otimes\calU_{\mu^k},1\right\rangle s_\muhat\right)=(q-1)\,\Log\,\Omega(q).
$$
\end{proof}

From Proposition \ref{expgen}, Theorem \ref{THEO1} and Theorem \ref{theogeneric}(2) we deduce the following one.

\begin{theorem} Let $\muhat\in\oP-\{0\}$. We have the following assertions.

\noindent (i) The polynomial $U_\muhat(t)$ has non-negative integer coefficients.

\noindent (ii) The polynomial $U_\muhat(t)$ is non-zero if and only if there exists $\tomhat=\alphahat_1^{n_1}\cdots\alphahat_s^{n_s}\in\otT$ and $(\tau^1,\dots,\tau^s)\in\calR_{\tomhat,\,\muhat}$ such that for all $i=1,\dots,s$ we have 

\begin{equation}
\ell(\tau^i)\leq V_{\alphahat_i}(1).
\label{in1}\end{equation}
\label{maintheo1}\end{theorem}

\begin{remark} Note that, by Theorem \ref{theogeneric} (1), the inequality (\ref{in1}) does not hold unless $\v_{\alphahat_i}$ is a root of $\Gamma_{\alphahat_i}$. 
\label{remroot}
\end{remark}

We are now going to give a simple sufficient condition for $U_\muhat(t)$ to be non-zero.

Denote by $\otTp$ the set of $\tomhat=\alphahat_1^{n_1}\cdots\alphahat_s^{n_s}\in\otT$ such that $\v_{\alphahat_i}$ is root of $\Gamma_{\alphahat_i}$.
By Proposition \ref{expgen}, \S  \ref{tech} and Remark \ref{remroot} we have a decomposition

\begin{equation}
U_\muhat(t)=\sum_{\tomhat\in\otTp}W_\muhat^\tomhat(t)
\label{decomp}\end{equation}
where 

$$
W_\muhat^\tomhat(t)=\sum_{\omhat\in\overline{\mathfrak{H}}{^{-1}}(\tomhat)}A_\omhat^oV_\omhat(t)c_\omhat^\muhat.$$

For $\tomhat=\alphahat_1^{n_1}\cdots\alphahat_s^{n_s}\in\otTp$, put 

$$
d_\tomhat:=\sum_{i=1}^sn_id_{\alphahat_i}
$$
where $d_\alphahat$, with $\alphahat\in\oP$, is given by Formula (\ref{dmu}).

For all $\omhat\in\overline{\mathfrak{H}}{^{-1}}(\tomhat)$, the degree of the polynomial $V_\omhat(t)$  is $d_\tomhat/2$ by Theorem \ref{theogeneric} (2).

By Remark \ref{inner} we have

$$
\sum_{\omhat\in\overline{\mathfrak{H}}{^{-1}}(\tomhat)}A_\omhat^oc_\omhat^\muhat=\left\langle \C_\tomhat^\muhat,1\right\rangle_{W_\tomhat}.
$$
We deduce the following proposition.

\begin{proposition} Let  $\tomhat=\alphahat_1^{n_1}\cdots\alphahat_s^{n_s}\in\otTp$. If $\left\langle \C_\tomhat^\muhat,1\right\rangle_{W_\tomhat}\neq 0$, then 
$$
\left\langle \C_\tomhat^\muhat,1\right\rangle_{W_\tomhat} q^{d_\tomhat}
$$
is the term of $W_\muhat^\tomhat(t)$ of highest degree. In particular, if for all $i=1,\dots,s$ the vector $\v_{\alphahat_i}$ is a real root (in which case $d_{\alphahat_i}=0$), then 

$$
W_\muhat^\tomhat(t)=\left\langle \C_\tomhat^\muhat,1\right\rangle_{W_\tomhat}.
$$
\label{suffcond}\end{proposition}

\begin{remark}Let $\muhat=(\mu^1,\dots,\mu^k)\in\oP$ and $\tomhat=\alphahat_1^{n_1}\cdots\alphahat_s^{n_s}\in\otT$ be of same size $n$. Notice that if $\alphahat_i=((1),\dots,(1))$ for all $i$, then $W_\tomhat=\calS_n$ and 

$$
\left\langle \C_\tomhat^\muhat,1\right\rangle=\langle H_\muhat,1\rangle
$$
where $H_\muhat$ is the $\C[\calS_n]$-module $H_{\mu^1}\otimes\cdots \otimes H_{\mu^k}$. 
Note also that if $n_1=n_2=\cdots=n_s=1$ then $W_\tomhat=1$ and 

$$
\left\langle \C_\tomhat^\muhat,1\right\rangle={\rm dim}\,\C_\tomhat^\muhat.
$$
Hence in this case $\left\langle \C_\tomhat^\muhat,1\right\rangle$ is a product of Littlewood-Richardson coefficients. If moreover $s=1$ and $n_1=1$, then 

$$\left\langle \C_\tomhat^\muhat,1\right\rangle=\delta_{\tomhat,\,\muhat}.
$$
\label{cases}\end{remark}

\begin{corollary}Let $\muhat\in\oP-\{0\}$. If there exists $\tomhat\in\otTp$ such that $\left\langle \C_\tomhat^\muhat,1\right\rangle\neq 0$, then $U_\muhat(t)\neq 0$. In particular, if $\v_\muhat$ is a root  or if $\left\langle H_\muhat,1\right\rangle\neq 0$, then $U_\muhat(t)\neq 0$.
\label{coro0}\end{corollary}

\begin{proof} By Theorem \ref{maintheo1} (i), the polynomials $W_\muhat^\tomhat(t)$ have non-negative integer coefficients and so there are no cancellation in the decomposition  (\ref{decomp}). We can now apply Proposition \ref{suffcond} to deduce that if $\left\langle \C_\tomhat^\muhat,1\right\rangle\neq 0$ for some $\tomhat$, then $U_\muhat(t)\neq 0$. \end{proof}

Note that Corollary \ref{coro0} is also a straightforward consequence of Theorem \ref{maintheo1}. Indeed if $\left\langle \C_\tomhat^\muhat,1\right\rangle\neq 0$ with $\tomhat=\alphahat_1^{n_1}\cdots\alphahat_s^{n_s}\in\otTp$, then the multi-partition $((n_1)^1,\dots,(n_s)^1)$ belongs to $\calR_{\tomhat,\,\muhat}$ and clearly $\ell((n_i)^1)=1\leq V_{\alphahat_i}(1)$ as $\alphahat_i$ is a root.

\begin{theorem}Let $\muhat\in\oP-\{0\}$. 

(i) If $\v_\muhat$ is a root of $\Gamma_\muhat$, then the degree of the polynomial $U_\muhat(t)$ is  at least $d_\muhat/2$. 

(ii) If  $\delta(\muhat)\geq 2$, then the degree of $U_\muhat(t)$  is exactly  $d_\muhat/2$. 

(ii) If $\delta(\muhat)\geq 3$ or $g=0$, $k=3$ and $\delta(\muhat)=2$, then $U_\muhat(t)$ is a monic polynomial.
\label{degtheo}\end{theorem}

\begin{remark} The degree of $U_\muhat(t)$ may be strictly larger than $d_\muhat/2$. Indeed, assume that $g=0$, $n=6$ and $k=3$, and take $\muhat=((2^3),(2^3), (2^3))$. Note that $\v_\muhat=2\cdot\alphahat_1$ where $\alphahat_1$ is the indivisible imaginary root of $\tilde{E}_6$. A direct calculation shows that $d_\muhat=2=d_{\alphahat_1}$. Consider $\tomhat=\alphahat_1^2$. Then a direct calculation, using that $V_{\alphahat_1}(t)=t$ (see Section \S \ref{ex}) shows that $W_\muhat^\tomhat(t)=t^2$ while $d_\muhat/2=1$.
\end{remark}

For $\muhat=(\mu^1,\dots,\mu^k)\in\oP$, put

$$
\Delta(\muhat):=\frac{1}{2}d_\muhat-1=\frac{1}{2}(2g-2+k)n^2-\frac{1}{2}\sum_{i,j}(\mu^i_j)^2.
$$
Notice that if $\alphahat=(\alpha^1,\dots,\alpha^k)\in\oP$ (possibly of size different from $|\muhat|$), then 

$$
2|\muhat|\Delta(\alphahat)=\delta(\muhat)|\alphahat|^2+\sigma_\muhat(\alphahat).
$$
where $\sigma_\muhat(\alphahat):=\sum_{i=1}^k\sigma_{\mu^i}(\alpha^i)$.

\begin{proposition} Let $\muhat,\alphahat_1,\dots,\alphahat_s\in\oP-\{0\}$ such that  $|\muhat|=|\sum_{i=1}^s\alphahat_i|$ and $\muhat\unlhd\sum_{i=1}^s\alphahat_i$.

(i) Assume that for all $i=1,\dots,s$, we have $\delta(\muhat)|\alphahat_i|\geq 2$. Then 

\begin{equation}
d_\muhat\geq \sum_i d_{\alphahat_i}\label{dimine}
\end{equation}
If moreoever  $\delta(\muhat)|\alphahat_i|>2$ for some $i=1,\dots,s$, then the inequality (\ref{dimine}) is strict.

(ii) Assume that $g=0$, $k=3$ and $\delta(\muhat)\geq 2$. Then the inequality (\ref{dimine}) is strict.
\label{harcoscoro}\end{proposition}

\begin{proof}Put $n:=|\muhat|$. Let us first prove (i).
$$
2n\Delta(\muhat)-2n\sum_{i=1}^s\Delta(\alphahat_i)=\delta(\muhat)\left(n^2-\sum_i|\alphahat_i|^2\right)+\sigma_\muhat(\muhat)-\sum_{i=1}^s\sigma_\muhat(\alphahat_i).
$$
We thus have

\begin{align*}
nd_\muhat-n\sum_{i=1}^sd_{\alphahat_i}&=2n\Delta(\muhat)-2n\sum_{i=1}^s\Delta(\alphahat_i)-2n(s-1)\\
&=\delta(\muhat)\left(\sum_{i\neq j}|\alphahat_i|\,|\alphahat_j|\right)-2n(s-1)+\sigma_\muhat(\muhat)-\sum_{i=1}^s\sigma_\muhat(\alphahat_i)\\
&=\sum_{i=1}^s\sum_{j\neq i}|\alphahat_i|\,\left(\delta(\muhat)|\alphahat_j|-2\right)+\sigma_\muhat(\muhat)-\sum_{i=1}^s\sigma_\muhat(\alphahat_i)
\end{align*}
By Theorem \ref{harcos} we have $\sigma_\muhat(\muhat)-\sum_{i=1}^s\sigma_\muhat(\alphahat_i)\geq 0$ hence the assertion (i).

We now prove (ii). If $|\alphahat|\leq  2$ then a straightforward calculation shows that $d_\alphahat\leq 0$, hence if for all $i=1,\dots,s$, we have $|\alphahat_i|\leq 2$, then clearly $d_\muhat-\sum_id_{\alphahat_i}$ is strictly positive. 

\end{proof}

\begin{proof}[Proof of Theorem \ref{degtheo}] By the decomposition (\ref{decomp}), we have

$$
U_\muhat(t)=V_\muhat(t)+\sum_{\tomhat\neq\muhat} W_\muhat^\tomhat(t).
$$

Recall that when $\v_\muhat$ is a root, the polynomial $V_\muhat(t)$ is non-zero monic  of degree $d_\muhat/2$ by Theorem \ref{theogeneric}. Hence (i). The assumption $\delta(\muhat)\geq 2$ implies (by Proposition \ref{harcoscoro}) that the degree of $W_\muhat^\tomhat(t)$ is smaller or equal to $d_\muhat/2$. Since the leading coefficients of $W_\muhat^\tomhat(t)$ are non-negative  we get assertion (ii).
The assertion (iii) is also a consequence of Proposition \ref{harcoscoro}. Indeed, in this case the degrees of $W_\muhat^\tomhat(t)$ are strictly smaller than $d_\muhat/2$.

\end{proof}

\begin{corollary} Denote by ${\rm St}_n$ the Steinberg character of $\GL_n(\F_q)$. Then $\langle {\rm St}_n\otimes{\rm St}_n\otimes{\rm St}_n,1\rangle$ is a monic polynomial of degree $\frac{1}{2}(n-1)(n-2)$ for all $n\geq 1$. 
\label{stein}\end{corollary}

\begin{proof} We have $\langle {\rm St}_n\otimes{\rm St}_n\otimes{\rm St}_n,1\rangle=U_\muhat(q)$ with $\muhat=((1^n),(1^n),(1^n))$. We find that $d_\muhat=(n-1)(n-2)$. Now $\delta(\muhat)=n-3$. Hence by Theorem \ref{degtheo} the corollary is true for $n\geq 5$. The cases $n=2,3,4$ are not difficult to work out with the above results (see \S \ref{ex} to see how to apply the above results) and we find that 

\begin{align*}
&\langle {\rm St}_2\otimes{\rm St}_2\otimes{\rm St}_2,1\rangle=1\\
&\langle {\rm St}_3\otimes{\rm St}_3\otimes{\rm St}_3,1\rangle=q+1\\
&\langle {\rm St}_4\otimes{\rm St}_4\otimes{\rm St}_4,1\rangle=q^3+2q+1\\
\end{align*}
Note that for $n\geq 3$, we have $\delta(\muhat)\geq 0$ and so $\v_\muhat$ is in the fundamental domain of imaginary root (see Proposition \ref{deltapos}). For $n=2$,  the vector $\v_\muhat$ is a real root of $D_4$.

\end{proof}

\subsection{Example} Assume that $n=k=3$ and $g=0$. We first need to list the $V_\alphahat(t)$ where $|\alphahat|\leq 3$ and $\v_\alphahat\in\Phi(\Gamma_\alphahat)$. This is given in the left column of the table below (up to permutation of the coordinates of $\alphahat$).
\bigskip

\begin{equation}
\label{table}
\begin{array}{|c|c|c|c|c|}
\hline
\alphahat&|\alphahat|&\Gamma_\alphahat & \v_\alphahat&V_\alphahat(t)\\
\hline
\alphahat_1:=((1^3),(1^3),(1^3))&3&\tilde{E}_6&  (3,(2,1),(2,1),(2,1))&t\\
\alphahat_{2,1}:=((2,1),(1^3),(1^3))&3&E_6 & (3, (1),(2,1), (2,1))&1 \\
\alphahat_3:=((1^2),(1^2),(1^2))&2&D_4 & (2,(1),(1),(1))&1 \\
\alphahat_4:=((1),(1),(1))&1&A_1 & (1)&1 \\
\hline
\end{array}
\end{equation}
\bigskip

\noindent where in the second column we only put the underlying graph of $\Gamma_\alphahat$ (as the orientation does not matter) and where $\v_\alphahat$ is written in the form $(v_0,(v_{[1,1]},v_{[1,2]},\dots),(v_{[2,1]},v_{[2,2]},\dots),\dots,(v_{[k,1]},v_{[k,2]},\dots))$. Note that $\v_{\alpha_{2,1}}$, $\v_{\alphahat_3}$ and $\v_{\alphahat_4}$ of the last three rows are real roots and so by Theorem \ref{theogeneric} we have $V_\alphahat(t)=1$ in these cases. Notice that $\v_{\alphahat_1}$ of the first row is the unique indivisible positive imaginary root of $\tilde{E}_6$. We can see that $V_{\alphahat_1}(t)=t$ either by computing directly $V_{\alphahat_1}(t)$ with Formula (\ref{genvuni2}) or by proceeding as follows. Put  $\alphahat_{2,2}:=((1^3),(2,1),(1^3))$ and $\alphahat_{2,3}:=((1^3),(1^3),(2,1))$ . Applying Formula (\ref{eq1}) we find that

\begin{align*}
A_{\alphahat_1}(t)&=V_{\alphahat_1}(t)+\sum_{i=1}^32\,V_{\alphahat_{2,i}}(t)\\
&=V_{\alphahat_1}(t)+6
\end{align*}
On the other hand the polynomial $A_{\alphahat_1}(t)$ is monic of degree $1$ by Theorem \ref{kactheo} and we know by \cite{crawley-boevey-etal} that the constant term of the polynomials $A_{\alphahat_1}(t)$ is the multiplicity of the root $\v_{\alphahat_1}$. Now the multiplicity of $\v_{\alphahat_1}$ is $6$ (see \cite[Chap. 7, Corollary 7.4]{Kac1}) and so $A_{\alphahat_1}(t)=t+6$. Hence $V_{\alphahat_1}(t)=t$.

The set of types $\tomhat\in\otTp$ of size $3$ are thus 
$$
\alphahat_1,\,\,\alphahat_{2,1},\,\,\alphahat_{2,2},\,\,\alphahat_{2,3},\,\, \alphahat_3\alphahat_4,\,\, \alphahat_4^3.$$
and so the decomposition (\ref{decomp}) reads

$$
U_\muhat(t)=W_\muhat^{\alphahat_1}(t)+W_\muhat^{\alphahat_{2,1}}(t)+W_\muhat^{\alphahat_{2,2}}+W_\muhat^{\alphahat_{2,3}}+W_\muhat^{ \alphahat_3\alphahat_4}(t)+W_\muhat^{\alphahat_4^3}(t).
$$
Since the dimension vectors associated to $\alphahat_{2,1},\,\alphahat_{2,2},\,\alphahat_{2,3},\, \alphahat_3,\, \alphahat_4,\,\alphahat_4$ are real roots, by Proposition \ref{suffcond} and Remark \ref{cases} we have 

\begin{align*}
U_\muhat(t)&=\delta_{\alphahat_1,\,\muhat}V_{\alphahat_1}(t)+\delta_{\alpha_{2,1},\,\muhat}+\delta_{\alphahat_{2,2},\,\muhat}+\delta_{\alphahat_{2,3},\,\muhat}+c_{\bf (1^2)(1)}^\muhat+\left\langle H_\muhat,1\right\rangle\\
&=\delta_{\alphahat_1,\,\muhat}t+\delta_{\alpha_{2,1},\,\muhat}+\delta_{\alphahat_{2,2},\,\muhat}+\delta_{\alphahat_{2,3},\,\muhat}+c_{\bf (1^2)(1)}^\muhat+\left\langle H_\muhat,1\right\rangle
\end{align*}
where $c_{\bf (1^2)(1)}^\muhat:=c_{(1^2)(1)}^{\mu^1}c_{(1^2)(1)}^{\mu^2}c_{(1^2)(1)}^{\mu^3}$ with $(\mu^1,\mu^2,\mu^3)=\muhat$.

Recall that $c_{(1^2)(1)}^\mu=1$ if $\mu\in\{(2,1),(1,1,1)\}$ and $c_{(1^2)(1)}^\mu=0$ otherwise.

Hence

\begin{equation*}
\label{table2}
\begin{array}{|c|c|c|c|c|c|}
\hline
\muhat&\delta_{\alphahat_1,\,\muhat}&\delta_{\alphahat_{2,1},\,\muhat}&c_{\bf (1^2)(1)}^\muhat&\left\langle H_\muhat,1\right\rangle&U_\muhat(t)\\
\hline
((3^1),(3^1),(3^1))&0&0&0&1&1\\
((3^1),(3^1),(2,1))&0&0&0&0&0\\
((3^1),(3^1),(1^3))&0&0 &0&0&0\\
((3^1),(2,1),(2,1))&0&0&0&1&1\\
((3^1),(2,1),(1^3))&0&0&0&0&0\\
((3^1),(1^3),(1^3))&0&0&0&1&1\\
((2,1),(2,1),(2,1))&0&0&1&1&2\\
((2,1),(2,1),(1^3))&0&0&1&1&2\\
((2,1),(1^3),(1^3))&0&1&1&0&2\\
((1^3),(1^3),(1^3))&1&0&1&0&t+1\\
\hline
\end{array}
\end{equation*}
\bigskip
\label{ex}

\end{document}